\documentclass[11pt,leqno]{amsart}
\usepackage{amsmath,amsfonts,latexsym,graphicx,amssymb,url, color}
\usepackage{txfonts} 

\usepackage[hmargin=2.5 cm,vmargin=2.5 cm]{geometry}

\usepackage{color} 
\newcommand{\R}{{\mathbb R}} %%reals

\newcommand{\A}{\mathbb{A}}

\newcommand\norm[1]{\left\| #1\right\|}

\newcommand{\I}{{\mathbb I}}

\newcommand{\wei}[1]{\langle #1 \rangle}
\newcommand{\sgn}{\text{sgn}}

\newtheorem{theorem}{Theorem}[section]
\newtheorem{definition}[theorem]{Definition}
\newtheorem{remark}[theorem]{Remark}
\newtheorem{lemma}[theorem]{Lemma}
\newtheorem{proposition}[theorem]{Proposition}
\newtheorem{corollary}[theorem]{Corollary}

\numberwithin{equation}{section}

\newcommand{\beq}{\begin{equation}}
\newcommand{\eeq}{\end{equation}}

\definecolor{darkred}{rgb}{.70,.12,.20}

\definecolor{darkgreen}{rgb}{.20,.52,.14}

\title[Weighted gradient estimates, degenerate Elliptic Equations] {Weighted-$W^{1,p}$ estimates for weak solutions of degenerate and singular elliptic equations}
\author{Dat Cao, Tadele Mengesha, and Tuoc Phan}
\address{Department of Mathematics, University of Tennessee, Knoxville, 227 Ayres Hall, 1403 Circle Drive, Knoxville, TN 37996, U.S.A.}
\email{dcao4@utk.edu, mengesha@math.utk.edu, phan@math.utk.edu}
\begin{document}
\begin{abstract} Global weighted $L^{p}$-estimates are obtained for the gradient of solutions to a class of linear singular, degenerate elliptic Dirichlet  boundary value problems over a bounded non-smooth domain. The  coefficient matrix is symmetric, nonnegative definite, and both its  smallest and largest eigenvalues are proportion to a weight in a Muckenhoupt class. 
Under a smallness condition on the mean oscillation of the coefficients with the weight and a Reifenberg flatness condition on the boundary of the domain, we  establish a weighted  gradient estimate for weak solutions of the equation. A class of degenerate coefficients satisfying the smallness condition is characterized.  A counter example to demonstrate the necessity of the smallness condition on the coefficients is given.  Our $W^{1,p}$-regularity estimates can be viewed as the Sobolev's counterpart of the  H\"{o}lder's regularity estimates  established by B. Fabes, C. E. Kenig, and R. P. Serapioni in 1982.
\end{abstract}

\maketitle
\section{Introduction} \label{Intro-sec}
%\subsection{Themes and discussions}
The main concern of this paper is to establish  a $W^{1,p}$-regularity estimate for weak solutions of the linear boundary value problem  
\begin{equation}  \label{main-eqn}
\left\{
\begin{array}{cccl}
\text{div} [\A(x)\nabla u] & = & \text{div}[{\bf F}] &\quad \text{in} \quad \Omega,\\
 u & = &0 & \quad \text{on} \quad \partial\Omega,
 \end{array}
\right.
\end{equation}
where  $\Omega \subset \mathbb{R}^n$ is an open bounded domain with boundary $\partial \Omega$, ${\bf F}:\Omega\rightarrow \mathbb{R}^{n}$ is a given vector field, and the coefficient matrix  $\A: \mathbb{R}^{n}~\rightarrow~\mathbb{R}^{n\times n}$ is  symmetric and measurable satisfying the degenerate elliptic condition
\begin{equation} \label{ellip}
\Lambda \mu(x) |\xi|^2 \leq \wei{\A(x) \xi, \xi} \leq \Lambda^{-1} \mu(x) |\xi|^2, \quad \forall \ \xi \in \mathbb{R}^n, \quad \text{a.e.} \quad x \in \mathbb{R}^n,
\end{equation}
with fixed  $\Lambda >0$, and a non-negative weight $\mu$ in some Muckenhoupt class. 
Our main result states that for a given $1< p < \infty$, the weak solution $u$ to \eqref{main-eqn} corresponding to ${\bf F}$ with ${\bf F}/\mu\in L^{p}(\Omega,\mu)$, the  weighted $L^{p}$ space,  satisfies the estimate
\begin{equation}\label{show-off-lp}
\|\nabla u\|_{L^{p}(\Omega, \mu)} \leq C \left\|\frac{{\bf F}}{\mu}\right\|_{L^{p}(\Omega,\mu)}, 
\end{equation}
provided that $\A$ has a small mean oscillation with weight $\mu$, and the boundary of $\Omega$ is sufficiently flat.  
We will demonstrate by an example that obtaining an estimate of type \eqref{show-off-lp} for  a solution of the degenerate elliptic equation \eqref{main-eqn} for large values of $p$ is not always possible even for a smooth degenerate coefficient matrix $\mathbb{A}$.  In light of the examples, this work provides the right set of conditions on the coefficients and on the boundary of $\Omega$ so that the linear map $\frac{{\bf F}}{\mu} \mapsto \nabla u $ is continuous on  $L^{p}(\Omega,\mu)$.

The study of regularity of weak solutions to  linear equations \eqref{main-eqn} when $\mathbb{A}$ is uniformly elliptic (i.e. for $\mu  =1$ in \eqref{ellip}) is by now classical.  The celebrated De Giorgi-Nash-M\"{o}ser theory \cite{DeGiorgi, Moser1, Moser, Nash}, for instance,  shows that weak solutions to \eqref{main-eqn} corresponding to uniformly elliptic coefficients are H\"{o}lder's continuous, when the datum ${\bf F}$ is sufficiently regular.   Regularity theory and related issues for the class of degenerate equations \eqref{main-eqn} with some weight $\mu$ were also investigated in past decades.  In this direction, seminal contributions were made in the classical papers \cite{Fabes, Murthy-Stamp}.  In particular, B. Fabes, C. E. Kenig, and R. P. Serapioni in \cite{Fabes} have established, among other significant results, the existence, and uniqueness of weak solutions in the weighted Sobolev space $W^{1,2}_0(\Omega, \mu)$ for $\mu$ in the Muckenhoupt class  $A_2$. In addition,  Harnack's inequality and H\"{o}lder's regularity of weak solutions were obtained in \cite{Fabes} by adapting the M\"{o}ser's iteration technique to the non-uniformly elliptic equation \eqref{main-eqn}. Since then,  H\"{o}lder's regularity theory of weak solutions for  linear, nonlinear degenerate elliptic and parabolic equations have been extensively developed in \cite{HKM, GW, MRW, MRW-1, NPS, Surnachev} by using and extending ideas and techniques in \cite{Fabes}. See also the earlier paper \cite{Str} on Gehring-type gradient estimate for solution of degenerate elliptic equations  

Sobolev type regularity theory for weak solutions of \eqref{main-eqn} have also been the focus of  studied in the past but mostly  for the uniformly elliptic case, i.e. $\mu =1$. In this case,  and unlike the case of H\"{o}lder's regularity, the mere assumption on the uniform ellipticity of the coefficients $\A$ is not sufficient for the gradient of the weak solution of \eqref{main-eqn} to have the same regularity as that of the data ${\bf F}$. This can be seen from the counterexample provided by 
N. G. Meyers in  \cite{M}. In the event that $\A$ is uniformly elliptic and continuous, the $L^p$-norm of $\nabla u$ can be controlled by the $L^p$-norm of the datum ${\bf F}$ and this is achieved via the Calder\'{o}n-Zygmund theory of singular integrals and a perturbation technique, see \cite{Evans, Krylov, Trud} for this classical  result. 
 The same approach was also used by \cite{CFL, Ragusa, Fazio-1,  Softova} to extend the  result 
when the coefficient matrix  $\A$  is uniformly elliptic and is in Sarason's class of vanishing mean oscillation (VMO) functions \cite{Sarason}. The approach in \cite{Fazio-1, Ragusa} is in fact based on the earlier work \cite{CFL} where many fundamental results on  Calder\'on-Zygmund operators were  established. A drawback of this approach is that it requires a Green's function representation of the solution to a corresponding elliptic equation used for comparison (usually a homogeneous equation with constant coefficients), which may not always be available for nonlinear equations. Alternative approaches have been  used in the papers \cite{CP, KZ1, Kim-Krylov} that avoid the use of singular integral theory directly but rather study the integrability of gradient of solutions, via approximation, as a function of the deviation of the coefficients from constant coefficients.  
See also the papers \cite{ B1, BW1, BW2, MP, MP-1, NP, Wang,LTT}, to cite a few,  for the implementation of these approaches for elliptic  and parabolic equations.

Unlike the case $\mu =1$,  estimates  of type \eqref{show-off-lp} for general $\mu \in A_2$ are not fully understood yet. Our goal in this paper, the first of several projects,  is to close this gap, by providing the right conditions on the coefficients $\A$ and the boundary of the domain $\Omega$ to obtain weighted gradient estimates for solutions of the degenerate elliptic problems \eqref{main-eqn} with \eqref{ellip} for $\mu \in A_2$. 
To establish \eqref{show-off-lp}, we follow  the approximation method of Cafarrelli and Peral in \cite{CP} where we view \eqref{main-eqn} locally as a perturbation of  an elliptic homogeneous equation with constant coefficients. The key to the success of this approach to degenerate equations is the novel way of measuring mean oscillation of coefficients that is  compatible with the degeneracy of the coefficients (see Definition \ref{BMOmu}).  As far as we know, this way of  measuring the mean oscillation of function relative to a given weight was first introduced in \cite{MW1, MW2} in connection with the study the Hilbert transform and the characterization of the dual of the weighted Hardy space. The condition we give on $\mathbb{A}$ is optimal in the sense that it coincided with the well known result in \cite{BW2} when $\mu = 1$. Via a counterexample we will also demonstrate the necessity of the smallness condition to obtain \eqref{show-off-lp}. A class of coefficients satisfying our smallness conditions will be given. Based on our approach and the recent developments \cite{BW1, BW2, MP, MP-1}, we are also able to obtain  estimates of type \eqref{show-off-lp} near the boundary of $\Omega$ for domain with a flatness condition on the boundary $\partial \Omega$. 

The paper is organized as follows. In Section \ref{Sec-main}, we introduce notations, definitions, and state the main results on the interior and global $W^{1,p}$-regularity estimates, Theorem \ref{local-grad-estimate-interior} and Theorem \ref{g-theorem}. An example, and a counterexample are also provided. 
Section \ref{preliminaries} recalls and proves several preliminary analytic results on weighted inequalities.  Necessary interior estimates and Theorem \ref{local-grad-estimate-interior} are  proved in Section \ref{approximation-est-sec}. Section \ref{boundary-global-section} gives the boundary approximation estimates  and completes the proof of Theorem \ref{g-theorem}. 
\section{Statements of main results and Examples}\label{Sec-main}
\subsection{Main results}
To state our main results, we need some notations and definitions.  We first introduce the notations that we use in the paper. Given a locally integrable function $\sigma\geq 0$, we denote by $d\sigma = \sigma dx$, a non-negative, Borel measure on $\mathbb{R}^n$. For  $U \subset \mathbb{R}^n$, a non-empty open set, we write 
\[\sigma(U) = \int_{U} \sigma(x) dx. \] 
For a locally integrable Lebesgue-measurable function $f$ on $\mathbb{R}^n$, we denote the average of $f$ in $U$ with respect to the measure $d\sigma$ as 
\[
\langle f\rangle_{\sigma, U}  = \fint_{U} f(x) d\sigma = \frac{1}{\sigma(U)} \int_{U}f(x) \sigma dx.
\]
In particular, with Lebesgue measure $dx$, we write
\[
\langle f\rangle_{U}  = \langle f\rangle_{dx, U} \quad \text{and} \quad  |U| = \int_{U} dx.
\] 
We now recall the definition of the class $A_p$ Muckenhoupt weights. 
For $p \in [1,\infty)$, the weight function $\mu \in L^1_{\textup{loc}}(\mathbb{R}^n)$  is said to be of class $A_p$ if 
\[
\begin{split}
[\mu]_{A_p}  &: = \sup_{ B\subset \mathbb{R}^{n}} \left( \fint_{B} \mu(y) dy \right)  \left(\fint_{B} \mu(y)^{-\frac{1}{p-1}} dy \right)^{p-1} < \infty,\quad 1 < p < \infty,\\
[\mu]_{A_1}  &: = \sup_{ B\subset \mathbb{R}^{n}} \left( \fint_{B} \mu(y) dy \right)  \|\mu^{-1}\|_{L^{\infty}(B)}< \infty,\quad p=1,
\end{split}
\]
where the supremum is taken over all balls $B \subset \mathbb{R}^n$.

Following \cite{Fabes, Murthy-Stamp}, for a given $\mu \in A_{p}$,  we can define the corresponding Lebesgue and Sobolev spaces with respect to the measure $d\mu$. 
For $1 \leq p < \infty$, we say a locally integrable function $f$ defined on $\Omega$ belongs to the weighted Lebesgue space $L^p(\Omega, \mu)$  if
\[\norm{f}_{L^p(\Omega,\mu)} =\left( \int_{\Omega} |f(x)|^p \mu(x) dx  \right)^{1/p}< \infty. \]
Let $k \in \mathbb{N}$. A locally integrable function $f$ defined on $\Omega$ is said to belong to the weighted Sobolev space $W^{k,p}(\Omega, \mu)$ if all of its distributional derivatives 
$D^\alpha f$ are in $L^p(\Omega, \mu)$ for $\alpha~\in~(\mathbb{N} \cup \{0\})^n$ with $|\alpha| \leq k$.  The space $W^{k,p}(\Omega,\mu)$ is equipped with the norm
\[
\norm{f}_{W^{k,p}(\Omega, \mu)} = \left( \sum_{|\alpha| \leq k} \norm{D^\alpha f}_{L^p(\Omega, \mu)}^p \right)^{1/p}.
\]
Moreover, we also denote $W^{1,p}_0(\Omega, \mu)$ to be  the closure of $C_0^\infty(\Omega)$ in $W^{1,p}(\Omega, \mu)$.

Now, we recall what we mean by weak solution of \eqref{main-eqn}.
\begin{definition} Assume that \eqref{ellip} holds and $|{\bf F}|/\mu \in L^p(\Omega, \mu)$ with $ 1 < p < \infty$. A function $u~\in~W^{1,p}_0(\Omega, \mu)$ is said to be a weak solution of \eqref{main-eqn} if
\begin{equation}\label{def-soln}
\int_\Omega \wei{\mathbb{A}\nabla u, \nabla \varphi} dx =  \int_\Omega \langle {\bf F}, \nabla \varphi\rangle dx, \quad \forall \varphi\in C_{0}^{\infty}(\Omega).  
\end{equation}
\end{definition} \noindent
To discuss about local interior regularity, we recall the following the definition of weak solution.
\begin{definition} \label{local-weak-solution} Assume that \eqref{ellip} holds and $|{\bf F}|/\mu \in L^p_{\textup{loc}}(\Omega, \mu)$ with $ 1 < p < \infty$. A function $u~\in~W^{1,p}_\textup{loc}(\Omega, \mu)$ is said to be a weak solution of 
\[
\textup{div}[\A \nabla u] = \textup{div}[{\bf F}], \quad \text{in} \quad \Omega
\]
if
\begin{equation*}
\int_\Omega \wei{\mathbb{A}\nabla u, \nabla \varphi} dx =  \int_\Omega \langle {\bf F}, \nabla \varphi\rangle dx, \quad \forall \varphi\in C_{0}^{\infty}(\Omega).  
\end{equation*}
\end{definition}
\noindent
The following definition of functions of bounded mean oscillations with weights introduced in \cite{MW1, MW2} will be  needed in our paper. 
\begin{definition}\label{BMOmu}
Given $R_{0}>0$, we say that a locally integrable function $f:\mathbb{R}^{n}\to \mathbb{R}$ is  a function of bounded mean oscillation with weight $\mu$ in $\Omega$ if 
\[
[f]_{\textup{BMO}_{R_{0}}(\Omega, \mu)}^2 = \sup_{\stackrel{x\in \Omega}{0 < \rho < R_{0}}} \frac{1}{\mu(B_{\rho}(x))} \int_{B_{\rho}(x)}|f(y) - \langle f\rangle_{B_{\rho}(x)}|^{2} \mu^{-1}(y) dy < \infty, 
\]
where $\langle f\rangle_{B_{\rho}(x)} = \frac{1}{|B_{\rho}(x)|} \int_{B_{\rho}(x)} f(y) dy$ is the average of $f$ in the ball $B_\rho(x)$. 
\end{definition}
Observe that this notion of bounded mean oscillation with weight is different from the weighted version of  the classical John-Nirenberg BMO, see \cite[eqn (1.2) and Theorem 5]{MW1}. However,  from this definition, the classical John-Nirenberg BMO space in $\Omega$ corresponds to $\mu=1$ and $R_{0} = \text{diam}(\Omega)$. 
\begin{definition} Let $\Lambda, R_0, \delta$ be given positive numbers, and let  $\mu\in A_{2}$. We denote
\[
\begin{split}
\mathcal{A}_{R_{0}} (\delta, \mu, \Lambda, \Omega)  := \bigg\{\mathbb{A}: & \mathbb{R}^{n}\to  \mathbb{R}^{n\times n} :  \text{$\mathbb{A}$  is measurable, symmetric such that} \\
& \eqref{ellip} \  \text{holds, and}\  [\A]^{2}_{\textup{BMO}_{R_{0}}(\Omega, \mu)} < \delta \bigg\}. 
\end{split}
\]
\end{definition} \noindent
In the above, for a given matrix function $\mathbb{A} = (a_{ij})$,  $[\A]^{2}_{\textup{BMO}_{R_{0}}(\Omega, \mu)} = \sum_{i,j=1}^{n} [a_{ij}]_{\textup{BMO}_{R_{0}}(\Omega, \mu)}^{2}$, where $[a_{ij}]_{\textup{BMO}_{R_{0}}(\Omega, \mu)}^{2}$ is as given in Definition \ref{BMOmu}.

The first main result of this paper is about the interior higher integrability of the gradients of weak solutions for the equation \eqref{main-eqn} which we state now. We use the notation $B_{r}$ for $B_{r}(0)$. 
\begin{theorem}\label{local-grad-estimate-interior}  Let $p\geq 2$, $M_0 \geq 1, \Lambda >0$, and let $\mu \in A_{2}$ such that $[\mu]_{A_{2}} \leq M_{0}$. There exists  a sufficiently small positive number $\delta= \delta(\Lambda, p, M_0,n)$  such that  if $\mathbb{A}\in \mathcal{A}_{4}(\delta, \mu, \Lambda, B_{2})$, 
${\bf F}/\mu\in L^{p}(B_{6}, \mu)$, and $u\in W^{1, 2}(B_{6}, \mu)$ is a weak solution of  
\[
\textup{div}[\mathbb{A} \nabla u] = \textup{div}({\bf F})\quad \text{in\ $B_{6}$},
\]
then $\nabla u\in L^{p}(B_{1}, \mu)$ and  
\[
\|\nabla u\|_{L^{p}(B_{1}, \mu)} \leq C \left( \mu(B_{1})^{\frac{1}{p} - \frac{1}{2}} \|\nabla u\|_{L^{2}(B_{6}, \mu)} + \|{\bf F}/\mu\|_{L^{p}(B_{6}, \mu)}\right),
\]
for some constant $C$ depending only on $\Lambda, p, n, M_{0}$. 
  \end{theorem}
Next, to obtain the global integrability for the gradients of weak solutions for \eqref{main-eqn}, we need to make precise the type of boundary the underlying domain $\Omega$ required to have. Intuitively,  we require that at all boundary points and at all scale,  locally, the boundary can be placed  between two hyperplanes.  
\begin{definition}
We say that $\Omega$ is a $(\delta, R_{0})$-Reifenberg flat domain if, for every $x\in \partial \Omega$ and every $r\in (0, R_{0})$, there exists a coordinate system $\{y_{1}, y_{2}, \cdots, y_{n}\}$ which may depend on $x$ and $r$, such that $x = 0$ in this coordinate system and that 
\[
B_{r}(0) \cap \{y_{n} > \delta r\} \subset B_{r}(0) \cap \Omega \subset B_{r}(0) \cap \{y_{n} > -\delta r\}. 
\] 
\end{definition}
We remark that, as described in \cite[Remark 3.2]{MP}, if $\Omega$ is a $(\delta, R_{0})$ flat domain with $\delta < 1$, then for any point $x$ on the boundary and $ 0 < \rho < R_{0}(1 -\delta)$,  there exists a coordinate system ${z_1,z_2,\cdots,  z_n}$  with the origin at some point in the interior of  $\Omega$ such that in this coordinate system $x = -\delta \rho z_n $ and
\[
B_{\rho}^{+}(0) \subset \Omega_{\rho}\subset B_{\rho}(0) \cap \{(z_{1}, \cdots, z_{n-1}, z_n): z_{n} > -2\delta' \rho\},\quad \text{with $\delta' = \frac{\delta}{1-\delta}.$}
\]
In the above and hereafter $B_{\rho}(x)$ denotes a ball of radius $\rho$ centered at $x$, $B_{\rho}^{+}(x)$ its upper-half ball, and $\Omega_{\rho}(x) = B_{\rho}(x)\cap \Omega$, the portion of the ball in $\Omega$. 
%%%%%%%%%%%%%%%%%

Our global regularity estimate for the weak solution of the equation \eqref{main-eqn} now can stated as below.
\begin{theorem} \label{g-theorem} Let $1< p < \infty$, $M_0 \geq 1$ and $\Lambda >0$.  There exists a sufficiently small $\delta=\delta (\Lambda, n, p, M_0)>0$ such that if $\Omega$ is $(\delta, R_{0})$ Reifenberg flat  and $\A \in \mathcal{A}_{R_{0}} (\delta, \mu, \Lambda, \Omega)$ for some $R_0 >0$, and some $\mu \in A_2 \cap A_p$ with  $[\mu]_{A_2 \cap A_p} \leq M_0$, then for each ${\bf F}: \Omega \rightarrow \mathbb{R}^n$ such that $|{\bf F}|/\mu \in L^p(\Omega, \mu)$, there exists unique weak solution $u \in W^{1,p}_0(\Omega, \mu)$ of \eqref{main-eqn}. Moreover, there is  some constants $C$ depending only on $n, \Lambda, p, M_0, R_0$ and $\text{diam}(\Omega)$ such that  
\begin{equation} \label{main-est}
\norm{\nabla u}_{L^p(\Omega, \mu)} \leq C 
\norm{\frac{{\bf F}}{\mu}}_{L^p(\Omega, \mu)}.
\end{equation}
 \end{theorem}
\noindent
Some comments regarding the results in Theorem~\ref{local-grad-estimate-interior} and Theorem~\ref{g-theorem} are in order. 
\begin{remark} 
\begin{itemize}
\item[\textup{(i)}] 
The weighted gradient regularity results in the above theorems  are a natural generalization of similar results obtained in  \cite{BW2, Fazio-1, Softova, Trud} for uniformly elliptic equations to equations of type \eqref{main-eqn} with degenerate/singular coefficients satisfying  \eqref{ellip}. 
The case $p =2$, the existence and uniqueness of weak solution of \eqref{main-eqn} in $W^{1,2}_0(\Omega, \mu)$ is already obtained in \cite{Fabes, Murthy-Stamp}.
\item[\textup{(ii)}] Equation \eqref{main-eqn} with \eqref{ellip} is invariant under the scaling: $\A \rightarrow \A/\lambda$, $\mu \rightarrow \mu/\lambda$, ${\bf F} \rightarrow {\bf F}/\lambda$, with $\lambda >0$. Therefore, by a simple  scaling argument, we see that the usual mean oscillation smallness condition on $\A$ in the classical John-Nirenberg \textup{BMO} norm, i.e.  the smallness requirement on $[\A]_{\textup{BMO}(\Omega, dx)}$, as in \cite{BW-no, B1, BW1, BW2, Fazio-1, LTT, MP-1, MP, NP} is not the right setting for  equation \eqref{main-eqn} with condition \eqref{ellip}.
\item[\textup{(iii)}]  Theorem \ref{g-theorem} will be proved first for the case $p>2$ and then use a duality argument for the case $1 < p < 2$. When $p > 2$, it is enough to assume $\mu \in A_{2}$ since $A_2 \cap A_p = A_2$ by the monotonicity of the Muckenhoupt classes. In this case, we already know that a unique solution in $W^{1,2}_0(\Omega, \mu)$ already exists by \cite{Fabes}. The main concern is thus  obtaining the estimate \eqref{main-est}. Once we have the estimate we may then apply \cite[Theorem 2.1.14]{Turesson} to conclude that $u\in W^{1,p}_0(\Omega, \mu)$. When $1 < p < 2$, the requirement on $\mu$ reduces to  being  in $A_{p}$ and is needed to apply Poincar\'e's inequality in the weighted space $W^{1,p}_{0}(\Omega,\mu)$. 
\end{itemize}
\end{remark}
Finally, we conclude this subsection by indicating that our implementation of the approximation method of Caffarelli and Peral in  \cite{CP} is influenced by the recent work \cite{B1, BW1, BW2, MP-1, MP, TN, NP, Wang}. The main idea in the approach is to locally consider the equation \eqref{main-eqn} as the perturbation of an equation for which the regularity of its solution is well understood. Key ingredients include Vitali's covering lemma, and the weak, strong $(p, p)$ estimates of the  weighted Hardy-Littlewood maximal operators. To be able to compare the solutions of the perturbed and un-pertured equations, we prefer to use compactness  argument as in \cite{BW-no, LTT, MP-1, MP, TN, NP}, but on weighted spaces, since this method could be more suitable when working with nonlinear equations as in \cite{TN, NP} and non-smooth domains as in \cite{BW-no, BW1, BW2, MP-1, MP}. Essential properties of $A_2$ weights such as reverse H\"{o}lder's inequality and doubling property are properly utilized in dealing with technical issues arising from the degeneracy and singularity of the coefficient $\A$. 
%===================================
\subsection{Counterexamples and examples} \label{examples-s}
This section contains two examples. The first example is a counterexample to demonstrate that solutions to degenerate homogeneous equations even with uniformly continuous coefficients $\A$ do not necessarily have gradient with high $\mu$-integrability.  This example also justifies the necessity of having the smallness of the mean oscillation with  $\mu$ for $\A$. The second example characterizes a class of coefficients for which our Theorem \ref{local-grad-estimate-interior} and Theorem \ref{g-theorem} apply. This example also provides the required rates of degenerate or singular of the coefficient $\A$ for the validity of the Sobolev's regularity theory of weak solutions of \eqref{main-eqn}.\\ \ \\
{\bf (i) A counterexample}: Let $n \geq 3$, $\alpha = \frac{1}{n+1}$, and $\mu(x)= |x|^{2(\alpha +1)}$ for $x \in \mathbb{R}^n$.  Note that since 
$n \geq 3$, we have $ 2 (\alpha +1)=\frac{2(n+2)}{n+1}~<~n$. Therefore, $\mu  \in A_2$. Also, with an $n\times n$ identity matrix $\I_n$, we consider  
\[ \mathbb{A}(x) = \mu(x) \I_n, \quad  u(x)=\frac{x_1}{|x|^{2\alpha}}, \quad x = (x_1, x_2, \cdots, x_n) \in  \Omega := B_1(0) \subset \mathbb{R}^n. \]
It is clear that $u \in L^1(\Omega) \cap L^2(\Omega, \mu)$. Moreover, by simple calculation, we see that the weak derivatives of $u$ are
\[ 
\begin{split}
u_{x_1} & = \frac{(1-2\alpha) x_1^2 +x_2^2 + \cdots + x_n^2}{|x|^{2(\alpha +1)}}, \quad \text{and} \\
u_{x_k} & = - \frac{2\alpha x_1 x_k }{|x|^{2(\alpha +1)}}, \quad k =2,3,\cdots, n. 
\end{split}
\]
A simple calculation also shows  that $u_{x_k} \in L^1(\Omega) \cap L^2(\Omega, \mu)$ for all $k =1,2,\cdots, n$ and that $u$ is a weak solution of
\[
\text{div}[\mathbb{A} \nabla u] =0, \quad \text{in} \quad \Omega.
\]
Indeed, for every $\varphi \in C_0^\infty(\Omega)$, we see that
\[
\begin{split}
\int_\Omega \wei{\mathbb{A}\nabla u, \nabla \varphi} dx & =  
\int_\Omega [(1-2\alpha) x_1^2 +x_2^2 + \cdots + x_n^2] \varphi_{x_1} dx - 2\alpha  \sum_{k=2}^n 
\int_{\Omega} x_1 x_k \varphi_{x_k} dx \\
& = - \Big[1 - (n+1) \alpha \Big]\int_\Omega \varphi (x)  x_1 dx =0.
\end{split}
\]
However, 
\[
\int_\Omega |\nabla u|^p \mu(x) dx \approx C(\alpha) \int_{B_1} |x|^{2\alpha + 2 - 2\alpha p} dx  
< \infty
\]
if and only if
\[
p <  \frac{2\alpha + n + 2}{2\alpha}.
\]
\begin{remark} In this example, $\A$ is uniformly continuous in $\overline{B}_1(0)$.  Therefore,  it is in the Sarason $\textup{VMO}(B_1(0))$ space.  In light of this and compared to \cite{Fazio-2}, Theorem \ref{local-grad-estimate-interior} and Theorem \ref{g-theorem} give the right conditions on $\A$ so that \eqref{main-est} holds. 
\end{remark}
\ \\
{\bf (ii) Examples of coefficients with small mean oscillation with weights}: In this example, we use 
the standard $A_2$ weight $\mu(x) = |x|^\alpha$ and $\A(x) = \mu(x) \I_n$. We show that if $|\alpha|$ is sufficiently small, then so is the mean oscillation of $\A$ with $\mu$. 
The proof of the next lemma is given in the appendix. 

\begin{lemma}\label{Example2-combinedLemma}
Let  $\mu (x) = |x|^{\alpha}$ for  $x \in \mathbb{R}^n$ and $|\alpha| \leq 1$. Then we have that 
\begin{itemize}
\item[(i)]  $\mu \in A_{2}$ and 
\[
[\mu]_{A_{2}} \leq M_{0} = M_{0}(n), and 
\]
\item[(ii)] \[
\int_{B_{r} (x_0)} \Big| \mu(x) - \langle \mu\rangle_{B_{r}(x_0)}\Big | dx \leq \frac{2|\alpha|4^{2n+1} }{2n-1}  \int_{B_r(x_0)} \mu(x) dx, \quad \forall x_0 \in \mathbb{R}^n, \quad \forall r >0.
\]
\end{itemize}

\end{lemma}
Now, for a given $\delta >0$, the next lemma shows that there is $\alpha_0 >0$ such that $\mu \in \A \in  \mathcal{A}_{R_{0}} (\delta, 1,\mu,B_{1})$ for any $\alpha\in (-\alpha_{0}, \alpha_0)$, and for every $R_0 >0$.
\begin{lemma} There exists a constant $C(n)$ such that if $|\alpha| \leq 1$, then
\[ [\A]_{\textup{BMO}(B_1(0), \mu)}^2\leq C(n) |\alpha|. \]
\end{lemma} 
\begin{proof} Observe from the proof of  \cite[Theorem 4]{MW1} for each $\mu \in A_2$ with $[\mu]_{A_2} \leq M_0$,  there exists a constant $C~=~C(M_0, n)$ such that
\begin{equation} \label{MW-ineq}
\begin{split}
& \sup_{B_r(x) \subset \mathbb{R}^n} \frac{1}{\mu(B_r(x))} \int_{B_r(x)} |\A(y) -\wei{\A}_{B_r(x)}|^2 \mu(y)^{-1} dy \\
 & \leq C(M_0, n) \sup_{B_r(x) \subset \mathbb{R}^n} \frac{1}{\mu(B_r(x))} \int_{B_r(x)} |\A(y) -\wei{\A}_{B_r(x)}| dy. 
 \end{split}
 \end{equation}
Therefore, it follows from Lemma \ref{Example2-combinedLemma}, and \eqref{MW-ineq} that if $|\alpha| \leq 1$, then 
\[ [\A]^{2}_{\textup{BMO}(B_1(0), \mu)}\leq C(n) |\alpha|.
\]
\end{proof}

\begin{remark} When $\alpha <0$, the coefficient  $\A = \mu\I_{n}$ is not in $L^p(B_1)$ for large $p$, and so it  does not belong to the standard John-Nirenberg \textup{BMO} space. \textup{Theorem} \ref{g-theorem}, therefore,   captures an important case that is not covered in many known work such as \cite{BW-no, B1, BW1, BW2, CFL, Fazio-1, LTT, MP-1, MP, NP, Softova} in which the requirement that $\A$ is sufficiently small in the John-Nirenberg $\textup{BMO}$ is essential.
\end{remark} 

\section{Preliminaries on weights and weighted norm inequalities} \label{preliminaries}
This section reviews and proves some basic results related to $A_{p}$ weights which are needed later for the proofs of Theorem \ref{local-grad-estimate-interior} and Theorem \ref{g-theorem}. We first state a result that follows from standard measure theory (see for example \cite[Lemma 3.12]{MP}).
\begin{lemma} \label{measuretheory-lp}
Assume that $g\geq 0$ is a measurable function in a bounded subset $U\subset \mathbb{R}^{n}$. Let $\theta>0$ and $\varpi>1$ be given constants. If $\mu$ is a weight in  $\mathbb{R}^{n}$, then for any $1\leq p < \infty$ 
\[
g\in L^{p}(U,\mu) \Leftrightarrow S:= \sum_{j\geq 1} \varpi^{pj}\mu(\{x\in U: g(x)>\theta \varpi^{j}\}) < \infty. 
\]
Moreover, there exists a constant $C>0$ such that 
\[
C^{-1} S \leq \|g\|^{p}_{L^{p}(U,\mu)} \leq C (\mu(U) + S), 
\]
where $C$ depends only on $\theta, \varpi$ and $p$. 
\end{lemma} \noindent
For a given locally integrable function $f$ we define the weighted Hardy-Littlewood maximal function as 
\[
\mathcal{M}^{\mu}f(x) = \sup_{\rho > 0}\fint_{B_{\rho}(x)}|f| d\mu  =  \sup_{\rho > 0} \frac{1}{\mu(B_{\rho} (x)) }\int_{B_{\rho}(x)}|f| \, \mu(x) dx. 
\]
For functions $f$ that are defined on a bounded domain, we define 
\[
\mathcal{M}_{\Omega}^{\mu}f(x) = \mathcal{M}^{\mu}(f\chi_{\Omega})(x).
\]
Recall the Muckenhoupt class $A_{p}$ defined in the previous section. For $1 < p < \infty,$ $A_{p}$ weights have a doubling property. For any $\mu \in A_{p}$, any ball $B$ and a measurable set $E\subset B$ we have that 
\begin{equation}\label{doubling}
\mu(B) \leq [\mu]_{A_{p}} \left(\frac{|E|}{|B|}\right)^{p} \mu(E)
\end{equation}
As a doubling measure, they also imply the boundedness of the Hardy-Littlewood maximal operator. Since we mostly use $A_2$-weights in this paper, we state the result for $A_2$-weights in the following lemma, which is a simpler version of a classical, more general result that can be found in \cite[Lemma $7.1.9$ - eqn (7.1.28) ]{Grafakos}.
\begin{lemma} Assume that $\mu \in A_2$ with $[\mu]_{A_2}\leq M_0$. Then, the followings hold.
\begin{itemize}
\item[(i)] Strong $(p,p)$: Let $1 < p < \infty$, then there exists a constant $C = C(M_0, n,p)$ such that  
\[ \|\mathcal{M}^{\mu}\|_{L^{p}(\R^n, \mu) \to L^{p}(\R^n, \mu)} \leq C. \] 
\item[(ii)] Weak $(1,1)$: 
There exists a constant $C=C(M_0, p, n)$ such that for any $\lambda >0$, we have 
\[
\mu(x\in \mathbb{R}^n: \mathcal{M} ^{\mu} (f) > \lambda) \leq \frac{C}{\lambda} \int_{\mathbb{R}^{n}}|f|d\mu.
\]  
\end{itemize}
\end{lemma}
From the definition of $A_2$-weights, it is immediate that $\mu \in A_2$, then so is $\mu^{-1}$ with 
\[ [\mu]_{A_2} = [\mu^{-1}]_{A_2}. \] 
$A_{2}$ weights satisfy the so called reverse H\"{o}lder's inequality. The statement of the following lemma and its proof can be found in \cite[Theorem 9.2.2 and Remark 9.2.3]{Grafakos}.
%==============
\begin{lemma} \label{A-2-reverse-Holder} 
For any $M_{0} > 0$, there exist  positive constants  $C = C (n, M_{0})$ and $\gamma = \gamma(n, M_{0})$ such that for all $\mu\in A_2$ satisfying $[\mu ]_{A_{2}} \leq M_{0}$, the following reverse H\"{o}lder conditions hold: 
\[
\begin{split}
& \left( \frac{1}{|B|} \int_{B} \mu^{(1+\gamma)}(x) dx \right)^{\frac{1}{1+\gamma}} \leq \frac{C}{|B|} \int_{B} \mu(x) dx,
\end{split}
\]
for every ball $B \subset \mathbb{R}^n$. In particular, the inequality is also valid for $\mu^{-1}$. 
\end{lemma}
Lemma \ref{A-2-reverse-Holder} implies the following inequalities  which will be used in this paper  frequently.
%==========
\begin{lemma} \label{L-p-L2mu} 
Let  $M_{0} > 0$, let $\gamma$  be  the constant as given in Lemma \ref{A-2-reverse-Holder}.  Then for any $\mu\in A_{2}$ satisfying $[\mu]_{A_{2}} \leq M_{0}$ and  any ball $B \subset  \mathbb{R}^n$ we have the following.
\begin{itemize}
 \item[(i)]  
If $u \in L^2(B, \mu)$, then  
$u \in L^{1+\beta}(B)$ where $\beta = \frac{\gamma}{2+ \gamma} >0$. 
Moreover, \[
\left (\fint_{B} |u|^{1+\beta} dx \right)^{\frac{1}{1+\beta}} \leq C(n, M_0) \left( \fint_{B} |u|^2 d\mu  \right)^{1/2}. 
\]
\item[(ii)] If $u \in L^q(B)$ for some $q \geq 1$, then $u\in L^{\tau}(B, \mu)$, where  $\tau =  \frac{q \gamma}{1+\gamma}. $  Moreover, 
\[
\left( \fint_{B}|u|^{\tau}d \mu \right)^{1/\tau} \leq C(n, M_0) \Bigg( \fint_{B} |u|^q dx \Bigg)^{1/q}. 
\]
\end{itemize}
\end{lemma}
\begin{proof}
 Both estimates follow from H\"older's inequality and Lemma \ref{A-2-reverse-Holder}.  We will demonstrate only (ii). 
\[
\begin{split}
\fint_{B}|u |^{\tau}d \mu  & \leq \frac{1}{\mu (B)} \Bigg( \int_{B} | u|^{\tau (1 + 1/\gamma)} dx \Bigg)^{\frac{\gamma}{1 + \gamma}} \Bigg (\int_{B} \mu ^{1 + \gamma}dx \Bigg)^{\frac{1}{1 + \gamma}}
 = \frac{1}{\mu(B)} \Bigg( \int_{B} |u|^{\tau (1 + 1/\gamma)} dx \Bigg)^{\frac{\gamma}{1 + \gamma}} \Bigg (\fint_{B} \mu ^{1 + \gamma}dx \Bigg)^{\frac{1}{1 + \gamma}} |B|^{\frac{1}{1 + \gamma}}.
\end{split}
\]
We now apply Lemma \ref{A-2-reverse-Holder} to obtain the estimate 
\[
\begin{split}
\fint_{B}|u |^{\tau}d \mu   & \leq C(n, M_{0})  \frac{1}{\mu(B)} \Bigg( \int_{B} |u|^{\tau (1 + 1/\gamma)} dx \Bigg)^{\frac{\gamma}{1 + \gamma}} \frac{\mu(B)}{|B|}|B|^{\frac{1}{1 + \gamma}} \\
&= C(n, M_{0}) |B|^{\frac{-\gamma}{1 + \gamma}} \Bigg( \int_{B} | u|^{\tau (1 + 1/\gamma)} dx \Bigg)^{\frac{\gamma}{1 + \gamma}} \\
&= C(n, M_0)  \Bigg( \fint_{B} | u|^q dx \Bigg)^{\tau/ q},
\end{split}
\]
and the proof of (ii) is complete.
\end{proof}
We remark that given $M_{0} > 0$, there exist constants $\varrho = \varrho(n, M_{0}) \in (0, 1)$ and $C = C(n,M_{0}) > 0$ such that for any ball $B\subset \mathbb{R}^{n}$, a measurable subset $E\subset B$ and any $\mu \in A_{2}$ with $[\mu]_{A_{2}} \leq M_{0}$ we have 
\begin{equation}\label{Ainfinity}
\mu(E) \leq C \left(\frac{|E|}{|B|}\right)^{\varrho} \mu(B). 
\end{equation}
This follows from Lemma \ref{L-p-L2mu} by taking $u = \chi_{E}$ in (ii). 

Next, we recall the weighted Sobolev-Poincar\'e inequality which can be found in \cite[Theorem 1.5]{Fabes}. 
\begin{lemma} \label{A_{2}-poincare} Let $M_0 >0$ and assume that $\mu \in A_2$ and $[\mu]_{A_2} \leq M_0$. Then, there exists a constant $C = C(n, M_0)$ and $\alpha  = \alpha (n, M_0)>0$ such that for every ball $B \subset \mathbb{R}^n$ of radius $r$, and every $u \in W^{1,2}(B, \mu)$, $ 1 \leq \varsigma \leq \frac{n}{n-1} + \alpha$, the following estimate holds
\[
\left(\frac{1}{\mu(B)} \int_{B} |u - A|^{2\varsigma} \mu(x) dx \right)^{\frac{1}{2\varsigma}} \leq C \,r\, \left(\frac{1}{\mu(B)} \int_{B} |\nabla u|^2 \mu(x) dx \right)^{1/2},
\]
where either
\[
A = \frac{1}{\mu(B)} \int_{B} u(x) d\mu(x), \quad \text{or} \quad A = \frac{1}{|B|} \int_{B} u(x) dx.
\]

\end{lemma}

Finally, we state a technical  lemma which is a consequence of Vitali's covering lemma. The proof can be found in \cite[Lemma 3.8]{MP-1}. 
\begin{lemma}\label{Vitali}
Let $\Omega$ be a $(\delta, R)$ Reifenberg flat domain with $\delta < 1/4$ and let $\mu$ be an $A_{p}$ weight for some $p>1$. 
Let $r_{0}>0$ be a fixed number and $C\subset D \subset \Omega$ be measurable sets for which there exists $0<\epsilon <1$ such that 
\begin{itemize}
\item[(i)]  $\mu(C) < \epsilon \mu(B_{r_{0}}(y)) $ for all $y\in \overline{\Omega}$, and 
\item[(ii)] for all $x\in \Omega$ and $\rho \in (0, 2r_{0}]$, if $\mu (C\cap B_{\rho}(x)) \geq \epsilon \mu(B_{\rho}(x))$, 
then $B_{\rho}(x)\cap \Omega \subset D. $
\end{itemize}  
Then we have the estimate 
\[
\mu(C) \leq \epsilon \left(\frac{10}{1 - 4\delta}\right)^{np} [\mu]^{2}_{p} \,\mu(D). 
\] 
\end{lemma}

%===============
%
\section{Interior $W^{1,p}$-regularity theory} \label{approximation-est-sec}
\subsection{Interior estimate setup} \label{interior-section} This section focuses on obtaining estimates for the gradient of solution to  
\begin{equation} \label{interior-eqn}
\text{div}[\mathbb{A}(x) \nabla u] = \text{div}[{\bf F}]  \quad \text{in} \quad B_{4}. 
\end{equation}
 For a weak solution $u \in W^{1,2}(B_4, \mu)$ of \eqref{interior-eqn}, our aim is to obtain estimates that approximate the gradient $\nabla u$ via a gradient of a solution to an associated homogeneous equation with constant coefficients.  To that end, we will find a constant elliptic, and symmetric matrix $\mathbb{A}_{0}$ sufficiently close to $\A(x)$  in an appropriate sense 
such that the weak solution $v$ of the equation
\begin{equation} \label{v-Q4} 
\begin{array}{ccll}
 \text{div}[ \mathbb{A}_{0} \nabla v]  &= & 0 & \quad \text{in} \quad B_4,
\end{array} 
\end{equation}
will be used in the comparison estimate. 
Recall that  $\mathbb{A}: B_6 \rightarrow \mathbb{R}^{n\times n}$ is measurable and symmetric satisfying the degenerate ellipticity condition: 
\begin{equation} \label{elip-interior}
\Lambda \mu(x) |\xi|^2 \leq \wei{\mathbb{A}(x) \xi, \xi} \leq  \mu(x)\Lambda^{-1}, \quad \text{for a.e.\ } x\in B_4, \quad \forall \ \xi \in \mathbb{R}^n. 
\end{equation}
for some fixed $\Lambda >0$ and $\mu \in A_2$. For a given $  M_0> 0$, we assume that 
\begin{equation} \label{M_0}
 [\mu]_{A_2} \leq M_0.
\end{equation}
Throughout the section, $\gamma >0$ is the number defined in Lemma \ref{A-2-reverse-Holder} which depends only on $M_0$ and $n$, and  let $\beta$ be as 
\begin{equation} \label{l-s.def}
\beta = \frac{\gamma}{2+\gamma} >0. 
\end{equation}
For now, we refer the readers to Definition \ref{local-weak-solution} for the definitions of weak solutions for the equations \eqref{interior-eqn} and \eqref{v-Q4}. The following well-known result on regularity for weak solutions of linear elliptic equations with constant coefficient  is also needed.
\begin{lemma} \label{reg-v} Let  $\mathbb{A}_{0}$ be an elliptic and symmetric constant $n\times n$ matrix  such that  there are positive numbers $\Lambda_0, \lambda_0$ such that
\[
\lambda_0 |\xi|^2 \leq \wei{\mathbb{A}_{0}\xi, \xi} \leq \Lambda_0 |\xi|^2, \quad \forall \ \xi \in \mathbb{R}^n.
\]
Then, if for some $1 < p < \infty$, $v \in W^{1,p}(B_{4})$ with  is a weak solution of \eqref{v-Q4}, 
then
\[
\norm{\nabla v}_{L^\infty(B_{7/2})} \leq C(n,p, \Lambda_0/\lambda_0)  \left [\fint_{B_{4}} |\nabla v|^p dx \right]^{1/p}.
\]
\end{lemma}
\begin{proof} Note that if $p \geq 2$, then $v$ is the energy solution and the lemma is the standard regularity result. On the other hand, if $ 1 < p < 2$, then it follows from \cite[Theorem 1]{Brezis} that the solution $v$  is in $W^{1,2}(B_r)$ for every $0 < r < 4$. From this, our lemma again follows by the classical regularity estimates, see \cite[Theorem 4.1]{Lin} for example.
\end{proof}
\subsection{Interior weighted Caccioppoli estimate }
The main result in this subsection is the following energy estimate for the difference $u - v$. 
%===================
\begin{lemma} \label{w-local-energy}  Let $\Lambda >0$, $M_{0}>0$ be given. 
Let $\mathbb{A}_{0}$ be an elliptic symmetric  constant matrix. Assume that \eqref{elip-interior} holds for some $\mu \in A_{2}$ with  $[\mu]_{A_{2}} \leq M_{0}$,  $u \in W^{1,2}(B_4, \mu)$ is a weak solution of \eqref{interior-eqn} and for some $q\in (1, \infty)$, $v \in W^{1,q}(B_4)$ is a weak solution of \eqref{v-Q4}. Define  $w = u - \wei{u}_{\mu, B_4}- v$. Then there exists a constant $C = C(n, \Lambda, M_0)$ such that  for any $\varphi \in C_{0}^{\infty}(B_{4})$, 
\[
\begin{split}
 \int_{B_4} |\nabla w|^2\varphi^2(x) d\mu  & \leq C \left[
\int_{B_4} \Big | \frac{{\bf F}}{\mu}\Big |^{2} \varphi^2 d \mu + (1   + \norm{\varphi \nabla v}_{L^\infty(B_4)}^2)\int_{B_4} w^2 |\nabla \varphi|^2 d\mu \right. \\
& \quad \quad + \left.
\norm{\varphi \nabla v}_{L^\infty(B_4)}^2  \int_{B_4} |\mathbb{A}(x) -\mathbb{A}_{0}|^2 \mu^{-1} dx
\right]. 
\end{split}
\]
\end{lemma}
\begin{proof} Note that since $\mathbb{A}_{0}$ is elliptic and constant, $v \in C^\infty_{\text{loc}}(B_4)$ as in Lemma \ref{reg-v}. Hence, $w \in W^{1,2}(\Omega', \mu)$ for every 
$\Omega' \subset\subset B_4$. Also, note that $w$ is a weak solution of the equation
\[
\text{div}[ \mathbb{A} \nabla w]  = \text{div}\Big[ {\bf F} - (\mathbb{A} - \mathbb{A}_{0}) \nabla v \Big] \quad \text{in} \quad B_4.
\]
By using  $w\varphi^2$ as a test function for this equation, we obtain
\[
\begin{split}
 \int_{B_4} \wei{\mathbb{A}\nabla w, \nabla w}\varphi^2 dx & = -\int_{B_4} \wei{\mathbb{A}\nabla w, \nabla (\varphi^2)} w dx   + \int_{B_4} \wei{{\bf F}, \nabla (w\varphi^2)} dx\\
 & \quad\quad - \int_{B_4} \wei{(\mathbb{A}-\mathbb{A}_{0}) \nabla v, \nabla(w\varphi^2)}dx. 
\end{split}
\]
We then have the following estimate 
{ {\begin{equation} \label{w-energy-test}
\begin{split}
\left| \int_{B_4} \wei{\mathbb{A}\nabla w, \nabla w}\varphi^2 dx \right| & \leq \left|\int_{B_4} \wei{\mathbb{A}\nabla w, \nabla (\varphi^2)} w dx \right| +  \int_{B_4} |{\bf F}| \left(|\nabla w| |\varphi|^2 + 2 |\nabla \varphi| |\varphi||w|  \right)dx \\
& + \int_{B_4} |(\mathbb{A}-\mathbb{A}_{0})| \left(|\nabla v| |\nabla w|  |\varphi|^2
+ 2|\nabla v| |\nabla \varphi| |w| |\varphi| \right)dx.
\end{split}
\end{equation}}}
Using the ellipticity condition \eqref{elip-interior},  we can estimate the term on the left hand side of \eqref{w-energy-test} as 
\[
 \Lambda \int_{B_4} |\nabla w|^2 \varphi^2(x) d\mu\leq \int_{B_4} \wei{\mathbb{A}\nabla w, \nabla w}\varphi^2 dx. 
\]
For $\epsilon >0$,  using \eqref{elip-interior} again, the first term on the right hand side can be estimated as  
\[
\begin{split}
\left|\int_{B_4} \wei{\mathbb{A}\nabla w, \nabla (\varphi^2)} w dx \right| &\leq 2 \Lambda^{-1} \int_{B_4} |\nabla w|  |\nabla \varphi| |w| \mu \varphi dx\\
&\leq \epsilon  \int_{B_4} |\nabla w|^2 |\varphi |^2 d\mu  + C(\Lambda, \epsilon) \int_{B_4} | \nabla \varphi|^2 |w|^2 d\mu 
\end{split}
\]
where we have applied  H\"{o}lder's inequality and Young's inequality.  
The second term on the right hand side in \eqref{w-energy-test} can be estimates as 
\[
\begin{split}
\int_{B_4} |{\bf F}| \left(|\nabla w| |\varphi|^2 + 2 |\nabla \varphi| |\varphi||w|  \right)dx& \leq   \int_{B_4} \frac{|{\bf F}|}{\mu} \left(|\nabla w| |\varphi|^2 + 2 |\nabla \varphi| |\varphi||w|  \right) \mu dx \\
& \leq \epsilon \int_{B_4} |\nabla w|^2 \varphi^2 d\mu + C(\epsilon) \left [ \int_{B_4} \Big |\frac{{\bf F}}{\mu}\Big|^2 \varphi^2  d\mu  +  \int_{B_4} w^2|\nabla \varphi|^2 d\mu \right].
\end{split}
\]
Finally, to estimate the third term in the right hand side of \eqref{w-energy-test}, we apply  H\"{o}lder's inequality  followed by Young's inequality as 
\[
\begin{split}
\int_{B_4} |\mathbb{A}-\mathbb{A}_{0}|& | \nabla v| \Big[ \varphi^2|\nabla w| + 2|w|\varphi|| \nabla \varphi| \Big] dx \\
 &\leq \norm{\varphi \nabla v }_{L^\infty(B_4)}\int_{B_4} |\mathbb{A}-\mathbb{A}_{0}|  \Big[ \varphi |\nabla w| + 2|w| | \nabla \varphi| \Big] \mu^{1/2} \frac{1}{\mu^{1/2}}dx\\
& \leq   \epsilon \int_{B_{4}}|\nabla w|^{2}\varphi^{2}(x)d\mu +  C(\epsilon) \norm{\varphi \nabla v }_{L^\infty(B_4)}^{2} \left[ \int_{B_{4}} |\mathbb{A} -\mathbb{A}_{0}|^{2} \mu ^{-1} dx + \int_{B_{4}} w^{2}|\nabla\varphi |^{2}d\mu\right].
\end{split}
\]
Then, collecting all the estimates and choosing $\epsilon$ sufficiently small to absorb the term containing $ \int_{B_4} |\nabla w|^2 \varphi^2 d\mu $ to the left hand side,  we obtain the desired result.
\end{proof}
\subsection{Interior gradient approximation estimates} 
%=================
Our first lemma confirms that we can approximate in $L^2(B_4, \mu)$ the weak solution $u \in W^{1,2}(B_4, \mu)$ of \eqref{interior-eqn} by a weak solution $v$ of \eqref{v-Q4} if the coefficient has small mean oscillation with weight  $\mu$ and the data ${\bf F}$ is sufficiently small relative to the weight.

\begin{lemma} \label{L2-aprox} Let $\Lambda >0, M_0 >0$ be fixed and let $ \beta$ be as in \eqref{l-s.def}. 
For every $\epsilon >0$ sufficiently small, there exists $\delta >0$ depending on only $\epsilon, \Lambda, n$, and $M_0$ such that the following statement holds true: If $\mathbb{A}, \mu, {\bf F}$ such that \eqref{elip-interior} and \eqref{M_0} hold, and 
\[ 
\frac{1}{ \mu(B_{4})} \int_{B_{4}} |\mathbb{A} -\langle\mathbb{A}\rangle_{B_{4}}|^{2} \mu ^{-1} dx   +  \frac{1}{ \mu(B_{4})} \int_{B_4} \Big | \frac{{\bf F}}{\mu} \Big |^2 d\mu (x)   \leq \delta^2,
\]
a weak solution $u \in W^{1,2}(B_4, \mu)$ of \eqref{interior-eqn} satisfies 
\begin{equation} \label{L2-gradient-u-int}
\fint_{B_4} |\nabla u|^2 d\mu \leq 1,
\end{equation}
then, there exists a constant matrix $\mathbb{A}_{0}$ and a weak solution $v \in W^{1,1+\beta}(B_4)$ of \eqref{v-Q4} such that
\[
\norm{\langle\mathbb{A}\rangle_{B_4} - \mathbb{A}_{0}} \leq \epsilon\ \frac{  \mu(B_{4})}{|B_{4}|},
\]
and
\[
\fint_{B_{7/2}} |\hat{u}  - v|^2 d\mu \leq \epsilon, \quad \text{where} \quad 
\hat{u} = u -\wei{u}_{\mu, B_{4}},  \quad
\wei{u}_{\mu, B_{4}} = \fint_{B_{4}} u(x) d\mu.
\]
Moreover, there is $C =C(\Lambda, n, M_0)$ such that
\begin{equation} \label{L2-gradient-v}
\fint_{B_{3}}|\nabla v|^2dx \leq C(\Lambda, n, M_0).
\end{equation}
\end{lemma}
\begin{proof} Note that for each $\lambda >0$, we can use the scaling $\mathbb{A}_{\lambda} = \frac{1}{\lambda}\mathbb{A}$, $\mu_\lambda = \mu/\lambda$ and ${\bf F}_\lambda = {\bf F}/\lambda$, then a weak solution $u$ of \eqref{interior-eqn}  will also be a weak solution to 
\[
\text{div}[\mathbb{A}_\lambda \nabla u] = \text{div} [{\bf F}_\lambda] \quad \text{in} \quad B_4.
\]
Moreover, $[\mu_\lambda]_{A_2} = [\mu]_{A_2}$ 
\[
\Lambda \mu_{\lambda}(x) |\xi|^2 \leq \wei{\mathbb{A}_\lambda(x) \xi, \xi} \leq \Lambda^{-1}\mu_{\lambda}(x)|\xi|^2, \quad \forall \ \xi \in \mathbb{R}^n, \quad \text{for a.e. \ } x \in B_4,
\]
and Lemma \ref{L2-aprox} is invariant with respect to this scaling. Therefore, without loss of generality, we can prove Lemma \ref{L2-aprox} with the additional assumption that
\begin{equation} \label{normalized-mu}
\wei{\mu}_{B_4} = \frac{1}{|B_4|}\int_{B_4} \mu(x) dx =1.
\end{equation}
In this case,  it follows from \eqref{elip-interior} and \eqref{normalized-mu} that 
\begin{equation} \label{bar-a-ellip}
\Lambda  |\xi|^2 \leq \wei{\langle \mathbb{A}\rangle_{B_4}\xi, \xi} \leq \Lambda^{-1} |\xi|^2, \quad \forall \ \xi \in \mathbb{R}^n.
\end{equation}

To proceed, we  use a contradiction argument. Suppose that there exists  $\epsilon_0 >0$ such that 
corresponding to $k \in \mathbb{N}$, there are $\mu_k \in A_2$, $ \mathbb{A}_k$  satisfying the degenerate ellipticity assumption as in \eqref{elip-interior} with $\mu$ and $\mathbb{A}$ are replaced by $\mu_k$ and $\mathbb{A}_k$ respectively,  and  ${\bf F}_k$ and  a weak solution $u_k \in W^{1, 2}(B_{4}, \mu_{k})$ of 
\begin{equation} \label{u-k.eqn}
 \text{div}[\mathbb{A}_k \nabla u_k]  = \text{div}({\bf F}_k) \quad \text{in }\, B_4,
\end{equation}
satisfying   
\begin{equation} \label{a_k} \left\{
\begin{split} 
&  \frac{1}{ \mu_k(B_{4})} \int_{B_4} |\mathbb{A}_k -\langle\mathbb{A}_{k}\rangle_{B_{4}}|^{2}\mu_{k}^{-1} dx  +  \fint_{B_4} \Big |\frac{{\bf F}_k}{\mu_k} \Big |^2 d\mu_k (x)  \leq \frac{1}{k^2}, \\
& [\mu_k]_{A_2} \leq M_0, \quad \wei{\mu}_{k, B_4} = \frac{1}{|B_4|} \int_{B_4} \mu_k(x) dx = 1,
\end{split} \right.
\end{equation}
with 
\begin{equation} \label{gradient-b-k}
  \fint_{B_4} |\nabla u_k|^2 d\mu_k \leq 1,
\end{equation}
but for all constant matrix $\mathbb{A}_{0}$ with
$
\|\langle\mathbb{A}_{k}\rangle_{B_{4}}- \mathbb{A}_{0} \| \leq \epsilon_{0},
$
 and all weak solution $v \in W^{1, 1+\beta}(B_4)$ of \eqref{v-Q4}, we have 
\begin{equation} \label{epsilon-0}
\fint_{B_{7/2}} |\hat{u}_{k} - v|^2 d\mu_k \geq \epsilon_0, \quad \text{with} \quad \hat{u}_k =u_k -\wei{u}_{k,\mu_k, B_{4}}.
\end{equation}
The sequence of constant matrices $\langle\mathbb{A}_{k}\rangle_{B_{4}} $ satisfies an estimate of the type \eqref{bar-a-ellip}, and therefore  the sequence $\langle\mathbb{A}_{k}\rangle_{B_{4}}$ is a bounded sequence  in $\mathbb{R}^{n\times n}$. Thus, by passing through a subsequence, we can assume that there is a constant matrix $\bar{\mathbb{A}}$ in $ \mathbb{R}^{n\times n}$ such that 
\begin{equation} \label{a-converge}
\lim_{k \rightarrow \infty} \langle\mathbb{A}_{k}\rangle_{B_{4}} = \bar{\mathbb{A}}.
\end{equation}
From \eqref{gradient-b-k},  and  Poincar\'e-Sobolev inequality Lemma \ref{A_{2}-poincare}, we see that
\[
\fint_{B_4} |\hat{u}_k|^2 d \mu_k  \leq C(n, M_{0}) \fint_{B_{4}} |\nabla u_k|^2 d\mu_k  \leq C(n, M_{0}),
\]
and therefore, for all $k\in \mathbb{N}$, $\norm{\hat{u}_k}_{W^{1,2}(B_4, \mu_k)} \leq C(n, M_0)$.  As a consequence, Lemma \ref{L-p-L2mu} implies that
\[
\norm{\hat{u}_k}_{W^{1,1 +\beta}(B_4)} \leq C(n, M_0) \norm{\hat{u}_k}_{W^{1,2}(B_4, \mu_k)} \leq C(n, M_0), \quad \beta = \frac{\gamma}{2 + \gamma} >0.
\]
Note that $\gamma$ is defined in Lemma \ref{A-2-reverse-Holder}, which only depends on $n$ and $M_0$. Therefore, by the compact imbedding $W^{1,1+\beta}(B_4)  \hookrightarrow L^{1+\beta}(B_4)$ and by passing through a subsequence, we can assume that there is $u \in W^{1,1+\beta}(B_4)$ such that
\begin{equation} \label{w-k-converge}
\left\{
\begin{split}
& \hat{u}_k \to u  \mbox{ strongly in } L^{1+\beta}(B_4),\quad \nabla u_k \rightharpoonup \nabla u \text{ weakly in } L^{1+\beta}(B_4),\
 \quad \text{and} \\
& \hat{u}_{k} \rightarrow u \ \text{a.e. in} \  \ B_4.
\end{split} \right.
\end{equation}
Moreover,
\begin{equation} \label{u-bound}
\norm{u}_{W^{1,1+\beta}(B_4)} \leq C(n, M_0).
\end{equation}
We claim that $u \in W^{1,1+\beta}(B_4)$ is a weak solution of 
\begin{equation} \label{w-unbounded-lambda}
 \text{div}[\bar{\mathbb{A}} \nabla u]  = 0  \quad \text{in} \quad  B_4. \\
\end{equation}
Let us fix a test function $\varphi \in C^\infty_0({B}_4)$. Then, by using $\varphi $ as a test function for the equation \eqref{u-k.eqn} of $u_k$, we have
\begin{equation} \label{test-uk-inter}
\int_{B_4} \wei{\mathbb{A}_k \nabla u_k, \nabla \varphi} dx = \int_{B_4} \langle {\bf F}_k, \nabla \varphi \rangle dx.
\end{equation}
We will take the limit $k \rightarrow \infty$ on both sides of the above equation. First of all, observe that by H\"{o}lder's inequality and \eqref{a_k}, it follows that the right hand side term of \eqref{test-uk-inter} can be estimated as 
\[
\begin{split}
\left|\fint_{B_4} {\bf F}_k \cdot \nabla \varphi  dx  \right| & \leq \left\{\fint_{B_4} \Big |\frac{{\bf F}_{k}}{\mu_k}\Big |^2 \mu_k dx \right\}^{1/2} \left\{\fint_{B_4} |\nabla \varphi|^2 \mu_k dx\right\}^{1/2} \\
& \leq \norm{\nabla \varphi}_{L^\infty(B_4)} \left\{\frac{1}{\mu_k (B_4)}\int_{B_4}\Big |\frac{{\bf F}_{k}}{\mu_k} \Big |^2d \mu_k (x) \right\}^{1/2}\frac{\mu_{k}(B_{4})}{|B_{4}|}\\
&  \leq \frac{\norm{\nabla \varphi}_{L^\infty(B_4)} }{k} . 
\end{split}
\]
Therefore, taking the limit as $k \to \infty,$ we have
\begin{equation} \label{lim-1}
\int_{B_4} \langle {\bf F}_k, \nabla \varphi  \rangle dx  =0. 
\end{equation}
On the other hand, it follows from  \eqref{a_k}, \eqref{gradient-b-k}, and H\"older's inequality that 
\[
\begin{split}
\left| \fint_{B_4} \wei{(\mathbb{A}_k - \langle\mathbb{A}_{k}\rangle_{B_{4}}) \nabla u_k, \nabla \varphi} dx \right| 
 & \leq \fint_{B_4} |\mathbb{A}_k - \langle\mathbb{A}_{k}\rangle_{B_{4}}| |\nabla u_k| \mu_k^{1/2} |\nabla \varphi| \mu_k^{-1/2} dx \\
& \leq \|\nabla \varphi\|_{L^{\infty}(B_{4})} \left\{\fint_{B_4} |\mathbb{A}_k - \langle\mathbb{A}_{k}\rangle_{B_4}|^{2} \mu_{k}^{-1}dx\right\}^{1/2} \left\{\frac{1}{|B_4|}\int_{B_4} |\nabla u_k|^2 d\mu_k \right\}^{1/2} \\
&\leq C(n, M_0)\frac{\norm{\nabla \varphi}_{L^\infty(B_4)}}{k}  \left\{\fint_{B_4} |\nabla u_k|^2 d\mu_k \right\}^{1/2}  \left\{\frac{1}{|B_4|} \int_{B_4} \mu_k(x) dx \right\}^{1/2} \\
& \leq C(n, M_0)\frac{\norm{\nabla \varphi}_{L^\infty(B_4)}}{k}   \rightarrow 0, \quad \text{as} \quad k \rightarrow \infty.
\end{split}
\]
As a result we have 
\[
0 = \lim_{k\to \infty }  \int_{B_4} \wei{(\mathbb{A}_k - \langle\mathbb{A}_{k}\rangle_{B_{4}}) \nabla u_k, \nabla \varphi} dx =  \lim_{k\to \infty } \Big[\int_{B_4} \wei{\mathbb{A}_k  \nabla u_k, \nabla \varphi} dx - \int_{B_4} \wei{ \langle\mathbb{A}_{k}\rangle_{B_{4}}  \nabla u_k, \nabla \varphi} dx\Big]. 
\]
We also observe that since $\nabla u_k$ converges weakly in $L^{1 + \beta}$ from \eqref{w-k-converge} and $\langle\mathbb{A}_{k}\rangle_{B_{4}} $ is a strongly converging sequence of constant symmetric matrices, we have that 
\[
\lim_{k\rightarrow \infty} \int_{B_4} \wei{\langle\mathbb{A}_{k}\rangle_{B_{4}}\nabla u_k, \nabla \varphi}dx =  \int_{B_{4}}\wei{\bar{\mathbb{A}}  \nabla u, \nabla\varphi} dx. 
\]
As a consequence we have that 
\begin{equation} \label{lim-2}
\lim_{k \rightarrow \infty} \int_{B_4} \wei{\mathbb{A}_{k} \nabla u_k, \nabla \varphi} dx = 
\int_{B_4} \wei{\bar{\mathbb{A}} \nabla u, \nabla \varphi} dx.
\end{equation} 
Combining  \eqref{lim-1} and \eqref{lim-2}, we see that
\[
 \int_{B_4} \wei{\bar{\mathbb{A}}\nabla u, \nabla \varphi} dx =0, \quad \forall \ \varphi \in C^\infty_0({B}_4).
\]
Now, from \eqref{bar-a-ellip}, and since $\bar{\mathbb{A}} = \lim_{k\rightarrow \infty} \langle\mathbb{A}_{k}\rangle_{B_{4}}$, we observe that
\[ 
\Lambda  |\xi|^2 \leq \wei{\bar{\mathbb{A}}\xi, \xi} \leq \Lambda^{-1} |\xi|^2, \quad \forall \ \xi \in \mathbb{R}^n.
\]
Hence, Lemma \ref{reg-v} implies that $u~\in~C^{\infty}(\overline{B}_{15/4})$. In addition, it follows from Lemma \ref{reg-v} and \eqref{u-bound} that
\begin{equation} \label{bound-u-mu-k}
\fint_{B_{7/2}} |\nabla u|^2 d\mu_k \leq \norm{\nabla u}_{L^\infty(B_{7/2})}^2 \leq C(n, \Lambda) \left( \fint_{B_4} |\nabla u|^{1+\beta} dx  \right)^{\frac{2}{1+\beta}}\leq C(n, M_0, \Lambda),  \quad \forall \ k \in \mathbb{N}.
\end{equation}
We claim that 
\[
\lim_{k\to \infty} \fint_{B_{7/2}}| \hat{u}_k - u - c_{k}|^{2} d\mu_{k} = 0, \quad \text{with} \quad  c_k = \fint_{B_{7/2}}[\hat{u}_k - u ] dx. 
\] 
To prove the claim, let us denote $H_{k} = \hat{u}_k - u - c_k$. Note that since $\mu_k \in A_2$ and $[\mu_k]_{A_2} \leq M_0$, it follows from \eqref{gradient-b-k} and doubling property of $\mu_{k}$ \eqref{doubling} that 
\[
\fint_{B_{7/2}} |\nabla u_k|^2 d\mu_k \leq \frac{\mu_k(B_4)}{\mu_k(B_{7/2})} \fint_{B_4} |\nabla u_k|^2 d\mu_k \leq \frac{\mu_k(B_4)}{\mu_k(B_{7/2})}   \leq C(n, M_0), \quad \forall \ k \in \mathbb{N}.
\]
This together with \eqref{bound-u-mu-k} yields
\begin{equation} \label{H-k-est}
\fint_{B_{7/2}}|\nabla H_k|^2 d\mu_k \leq C(n, \Lambda, M_0), \quad \forall \ k \in \mathbb{N}.
\end{equation}
On the other hand, by the weighted Sobolev-Poincar\'{e} inequality \cite[Theorem 1.5]{Fabes}, Lemma \ref{A_{2}-poincare}, there exists  $\varsigma >1$ such that
\[
 \Bigg(\fint_{B_{7/2}} |H_{k}|^{2 \varsigma }d\mu_{k}\Bigg)^{\frac{1}{2\varsigma }} \leq C (n, M_{0}) \Bigg(\fint_{B_{7/2}} |\nabla H_{k}|^{2}d\mu_{k}\Bigg)^{\frac{1}{2}}.
\] 
Let $\tau > 0$ be a small number that will be determined later. 
By H\"older's inequality, we have 
\[
\Bigg(\fint_{B_{7/2}} |H_{k}|^{2}d \mu_{k} \Bigg)^{1/2} \leq \Bigg(\fint_{B_{7/2}} |H_{k}|^{\tau}d\mu_{k}\Bigg)^{\theta/\tau} \Bigg(\fint_{B_{7/2}} |H_{k}|^{2 \varsigma }d\mu_{k}\Bigg)^{\frac{1-\theta}{2\varsigma }},
\]
with
\[
\theta = \frac{ \frac{1}{2} - \frac{1}{2\varsigma} }{\frac{1}{\tau}- \frac{1}{2\varsigma}} \in (0,1).
\]
Therefore,
\begin{equation} \label{holder-with-tau}
\begin{split}
\Bigg(\fint_{B_{7/2}} |H_{k}|^{2}  d\mu_{k} \Bigg)^{1/2} &\leq \Bigg(\fint_{B_{7/2}} |H_{k} |^{\tau}d\mu_{k}\Bigg)^{\theta/\tau} \Bigg(\fint_{B_{7/2}} |H_{k}|^{2 \varsigma }d\mu_{k}\Bigg)^{\frac{1-\theta}{2\varsigma }}\\
&\leq   (C (n, M_{0}))^{1 - \theta} \Bigg(\fint_{B_{7/2}} |H_{k}|^{\tau}d\mu_{k}\Bigg)^{\theta/\tau} \Bigg(\fint_{B_{7/2}} |\nabla H_{k}|^{2}d\mu_{k}\Bigg)^{\frac{1-\theta}{2}}.
\end{split}
\end{equation}
On the other hand, by Lemma \ref{L-p-L2mu}, we have 
\begin{equation}\label{reverse-holder-tau}
\begin{split}
\fint_{B_{7/2}}|H_{k}|^{\tau}d \mu_{k} & \leq C(n, M_{0}) \Bigg( \fint_{B_{7/2}} | H_{k}|^{\tau (1 + 1/\gamma)} dx \Bigg)^{\frac{\gamma}{1 + \gamma}}. 
\end{split}
\end{equation}
Moreover, observe that 
\[
\begin{split}
\fint_{B_{7/2}} |H_{k}|^{\tau (1 + 1/\gamma)} dx& \leq \fint_{B_{7/2}} |\hat{u}_{k} -u|^{\tau (1 + 1/\gamma)} dx + \fint_{B_{7/2}} |c_k|^{\tau (1 + 1/\gamma)} dx\\
&\leq 2 \fint_{B_{7/2}} |\hat{u}_{k} - u|^{\tau (1 + 1/\gamma)} dx. 
\end{split}
\]
Now choose $\tau $ small that $\tau (1 + 1/\gamma) \leq 1 + \beta$. We can then apply the strong convergence of $\hat{u}_{k} \to u $ in $L^{1 + \beta}(B_{4})$ as in \eqref{w-k-converge}, to conclude that 
\begin{equation}\label{strong-cone-tau}
 \fint_{B_{7/2}} |H_{k}|^{\tau (1 + 1/\gamma)} dx \to 0 \quad \quad \text{as $k \to \infty$.}
\end{equation}
Combining inequalities \eqref{H-k-est}, \eqref{holder-with-tau}, \eqref{reverse-holder-tau} and \eqref{strong-cone-tau} we obtain that 
\[
 \lim_{k\to \infty}\fint_{B_{7/2}} |H_{k}|^{2}d \mu_{k}  = 0.\]
This assertion proves our claim.  However, note that since $\langle\mathbb{A}_{k}\rangle_{B_{4}} \rightarrow \bar{\mathbb{A}}$ and $\epsilon_0 >0,$
\[
\norm{\langle\mathbb{A}_{k}\rangle_{B_{4}} - \bar{\mathbb{A}}} \leq \epsilon_0 
\]
for sufficiently large $k$. But, this contradicts to \eqref{epsilon-0} if we take $\mathbb{A}_{0} = \bar{\mathbb{A}}$, $ v = u - c_k$ and $k$ sufficiently large.

We finally prove the estimate \eqref{L2-gradient-v}. 
We assume the existence of $\beta,$ $\mathbb{A}_{0}$ and $v \in W^{1,1+\beta}(B_4)$ satisfying the first part of the lemma. Then, with sufficiently small $\epsilon$, we can assume that 
\[
\frac{\Lambda \mu(B_4) }{2|B_4|} |\xi|^2 \leq \wei{\mathbb{A}_{0} \xi, \xi} \leq 2\frac{ \Lambda^{-1} \mu(B_4)}{|B_4|} |\xi|^2, \quad \forall \ \xi \in \mathbb{R}^n.
\]
Hence, by the standard regularity theory for elliptic equations, Lemma \ref{reg-v}, $v$ is in $C^\infty(B_{15/4})$. Moreover, from standard regularity theory, we also have
\[
\fint_{B_{16/5}} |v|^2 dx  \leq  C(n, \Lambda) \left\{\fint_{B_{7/2}} |v|^{1 + \beta}  dx\right\}^{\frac{2}{1+\beta}}, \quad \text{with} \quad \beta = \frac{\gamma}{2+\gamma}.
\]
Then, it follows from Lemma \ref{L-p-L2mu} that
\[
\fint_{B_{16/5}} |v|^2 dx \leq C(n, \Lambda, M_0) \fint_{B_{7/2}} |v|^2 d\mu.
\]
From this last estimate and the energy estimate for $v$, we infer that
\[
\fint_{B_{3}} |\nabla v|^2 dx  \leq C(n, \Lambda, M_0) \fint_{B_\frac{16}{5}} |v|^2 dx \leq C(n, \Lambda, M_0) \fint_{B_{7/2}} |v|^2 d\mu.
\]
Therefore, 
\[
\begin{split}
\fint_{B_3} |\nabla v|^2 dx & \leq C(n, \Lambda, M_0) \left[ \fint_{B_{7/2}} |\hat{u}  - v|^2 d\mu + \fint_{B_{7/2}} |u - \wei{u}_{\mu, B_{4}} | ^2 d\mu  \right] \\ 
& \leq C(n, \Lambda, M_0) \left[\epsilon + \frac{\mu(B_4)}{\mu(B_{7/2})} \fint_{B_{4}} |u - \wei{u}_{\mu, B_{4}} | ^2 d\mu \right]\\ 
& \leq C(n, \Lambda, M_0) \left[\epsilon + \frac{\mu(B_4)}{\mu(B_{7/2})}  \fint_{B_4} |\nabla u|^2 d\mu \right],
\end{split}
\]
where we have used the Poincar\'e's inequality for weighted Sobolev spaces, Lemma \ref{A_{2}-poincare}. Since $\epsilon$ is small,  we can assume that $\epsilon <1$. It then follows from \eqref{L2-gradient-u-int} and \eqref{doubling}  that
\[
\fint_{B_3} |\nabla v|^2 dx  \leq C(n, \Lambda, M_0) \left[ 1 +  \frac{\mu(B_4)}{\mu(B_3)} \fint_{B_4} |\nabla u|^2 d\mu \right] \leq C(n, \Lambda, M_0).
\]
This assertion proves the estimate \eqref{L2-gradient-v} and completes the proof of the lemma. 
\end{proof}
%==================
The following main result of the section provides the approximation for the gradient of solution by gradient of a homogeneous equation with constant coefficients that is appropriately chosen. 
\begin{proposition} \label{L2-gradient-aprox} Let $\Lambda >0, M_0 >0$ be fixed and let $ \beta$ be as in \eqref{l-s.def}.
 For every $\epsilon >0$ sufficiently small, there exists $\delta >0$ depending on only $\epsilon, \Lambda, n, M_0$ such that the following statement holds true: If \eqref{elip-interior} and \eqref{M_0} hold and 
\[ 
 \frac{1}{ \mu(B_{4})}\int_{B_4} |\mathbb{A} -\langle\mathbb{A}\rangle_{B_4}|^{2} \mu^{-1} dx  +   \fint_{B_4} \Big |\frac{{\bf F}}{\mu} \Big |^2  d \mu(x)  \leq \delta^{2}, \]
for  every weak solution $u \in W^{1,2}(B_4, \mu)$ of \eqref{interior-eqn} satisfying
\[ \fint_{B_4} |\nabla u|^2 d\mu \leq 1, 
\]
then, there exists a constant matrix $\mathbb{A}_{0}$ and a weak solution $v \in W^{1,1+\beta}(B_4)$ of \eqref{v-Q4} such that
\[
|\langle\mathbb{A}\rangle_{B_4} - \mathbb{A}_{0}| \leq \frac{\epsilon\mu(B_{4})}{|B_{4}|}  , \quad \text{and} \quad
\fint_{B_{2}} |\nabla u -\nabla v|^2 d\mu \leq \epsilon. 
\]
Moreover, there is $C =C(\Lambda, n, M_0)$ such that
\begin{equation} \label{L2-gradient-v-aprox}
\fint_{B_3}|\nabla v|^2dx \leq C.
\end{equation}\end{proposition}
\begin{proof}  Let $\alpha >0$ sufficiently small to be determined. By Lemma \ref{L2-aprox}, there exists $\delta_1 >0$ such that if
\[ 
 \frac{1}{ \mu(B_{4})} \int_{B_4} |\mathbb{A} -\langle\mathbb{A}\rangle_{B_4}|^{2} \mu^{-1} dx  +  \fint_{B_4} \Big |\frac{{\bf F}}{\mu} \Big |^2  d \mu(x)  \leq \delta_{1}^2, \]
and if $u$ is a weak solution of \eqref{interior-eqn} satisfying
\[ \fint_{B_4} |\nabla u|^2 d\mu \leq 1,
\]
there exist a constant matrix $\mathbb{A}_{0}$ and a weak solution $v$ of \eqref{v-Q4} such that 
\begin{equation} \label{u-v-aproximat}
|\mathbb{A}_{0} -\langle\mathbb{A}\rangle_{B_4}| \leq  \alpha\frac{\mu(B_{4})}{|B_{4}|}  , \quad \fint_{B_{7/2}}|\hat{u} - v|^2 d\mu \leq \alpha, \quad  \text{and} \quad \fint_{B_3} |\nabla v|^2 dx \leq C(\Lambda, n, M_0).
\end{equation}
From \eqref{u-v-aproximat} and Lemma \ref{reg-v}, we conclude that
\begin{equation} \label{v-int-L-infty}
\norm{\nabla v}_{L^\infty(B_{\frac{5}{2}})} \leq C(n, \Lambda, M_0).
\end{equation}
Also, without loss of generality, we can assume that $\delta_1^2 \leq \alpha$. Hence we have that 
\[
\begin{split}
 \frac{1}{\mu(B_{4})}\int_{B_4} |\mathbb{A} -\mathbb{A}_{0}|^{2} \mu^{-1} dx& \leq  \frac{2}{\mu(B_{4})}\int_{B_4} |\mathbb{A} -\langle\mathbb{A}\rangle_{B_4}|^{2} \mu^{-1} dx + \frac{2}{\mu(B_{4})} |\mathbb{A}_{0} -\langle\mathbb{A}\rangle_{B_4}|^{2}  \int_{B_{4}}\mu^{-1}dx\\&\leq 2\delta_{1}^2+ 2 M_{0} \alpha^{2} \leq \alpha C 
\end{split}\]
From this last estimate, and the estimates \eqref{u-v-aproximat}-\eqref{v-int-L-infty}, and by applying Lemma \ref{w-local-energy}, we obtain
\[
\fint_{B_{2}} |\nabla u - \nabla v|^2 d\mu \leq \alpha \ C.
\]
where $C $ depends only on $n, \Lambda,$ and $M_{0}$. 
Thus, if we choose $\alpha$ such that $\epsilon =  \alpha \ C$, the assertion of lemma follows with $\delta = \delta_1$.
\end{proof}

%===================
\subsection{Proof of the interior $W^{1,p}$-regularity estimates} \label{density-est-sec}

We prove Theorem \ref{local-grad-estimate-interior} after establishing several estimates for upper-level set  of the maximal function $\mathcal{M}^{\mu}(\chi_{B_{6}}|\nabla u|^{2})$.  We begin with the following lemma. 
\begin{lemma} \label{varpi-lemma}
Suppose that $M_{0}>0$ and $\mu \in A_{2}$ such that $[\mu]_{A_{2}} \leq M_{0}$. 
There exists a constant $\varpi = \varpi(n, \Lambda, M_0)> 1$ such that the following holds true. Corresponding to any $\epsilon > 0 $, there exists a small constant $\delta= \delta(\epsilon, \Lambda, M_0, n)$ such that if 
$\mathbb{A}\in \mathcal{A}_{4}(\delta, \mu, \Lambda, B_{1})$,  $u\in W^{1, 2}(B_{4}, \mu)$ is a weak solution to 
\[
\textup{div}[\mathbb{A} \nabla u] = \textup{div}({\bf F})\quad \text{in $B_{4}$}, \quad \text{and}
\]
\begin{equation}\label{max-f-small-forrho1}
B_{1}\cap \{x\in \mathbb{R}^{n}: \mathcal{M}^{\mu}(\chi_{B_{6}}|\nabla u|^{2}) \leq 1 \}\cap \{x\in \mathbb{R}^{n}: \mathcal{M}^{\mu}\left(\left|\frac{{\bf F}}{\mu}\right|^{2}\chi_{B_{6}}\right) \leq \delta^{2} \} \neq \emptyset, 
\end{equation}
then 
\[
\mu(\{x\in \mathbb{R}^{n}: \mathcal{M}^{\mu}(\chi_{B_{6}}|\nabla u|^{2})  > \varpi^{2}\}\cap B_{1}) < \epsilon \mu(B_{1}). 
\]
\end{lemma}
\begin{proof} With a given $\epsilon > 0$, let $\eta >0$ to be chosen later, which is sufficiently small and is dependent only on $\epsilon$. Using this $\eta$,  Lemma \ref{reg-v}, and Proposition \ref{L2-gradient-aprox}, we can find  $\delta = \delta(\eta, \Lambda, M_0, n)> 0$ such that if $u$ is a weak solution,  
\begin{equation}\label{cond-lemma}
\frac{1}{\mu(B_{4})} \int_{B_{4}} |\mathbb{A} - \langle\mathbb{A}\rangle_{B_{4}}|^{2} \mu^{-1}dx  \leq \delta^2, \quad \fint_{B_{4}} \Big | \frac{{\bf F}}{\mu}\Big |^2 d\mu(x) \leq \delta^2, \text{and}\quad \fint_{B_{4}} |\nabla u|^{2}d\mu \leq 1
\end{equation}
then there exists a constant matrix $\mathbb{A}_{0}$ and a weak solution $v$ to $\text{div} ( \mathbb{A}_{0}\nabla v )= 0$ in $B_{4}$
satisfying 
\begin{equation}\label{conclusion-lemma}
|\langle\mathbb{A}\rangle_{B_{4}} - \mathbb{A}_{0}| < \eta\, \frac{\mu(B_{4})}{|B_{4}|},\,\, \fint_{B_{2}}|\nabla u-\nabla v| ^{2} d\mu< \eta
\quad \text{and\, $\| \nabla v\| _{L^{\infty}(B_{2})} \leq C_{0}$},
\end{equation}
 for some positive constant $C_{0}$ that depends only $n, \Lambda$ and $M_{0}$.
 
Next, by using this $\delta$ 
in the assumption \eqref{max-f-small-forrho1}, we can find $x_{0}\in B_{1}$ such that for any $r > 0$ 
\begin{equation}\label{level-set-eq1}
\fint_{B_{r}(x_{0})} \chi_{B_{6}}|\nabla u|^{2} d\mu(x) \leq 1,\quad \text{and } \fint_{B_{r}(x_{0})} \left|\frac{{\bf F}}{\mu}\right|^{2} \chi_{B_{6}} d\mu(x) \leq \delta^{2}. 
\end{equation}
We now  make some observations. 

\noindent {\it First,} we see that $B_{4} \subset B_{5}(x_{0})\subset B_{6}$ and therefore we have from \eqref{level-set-eq1} and \eqref{doubling} that 
\[
\fint_{B_{4}} |\nabla u|^{2} d\mu(x)\leq \frac{\mu(B_{5}(x_0))}{\mu(B_{4})}\fint_{B_{5}(x_0)} \chi_{B_{6}}|\nabla u|^{2} d\mu(x) \leq M_{0}\left(\frac{5}{4}\right)^{2n}, 
\]
and similarly 
\[
 \fint_{B_{4}} \left|\frac{{\bf F}}{\mu}\right|^{2}d\mu(x) \leq M_{0}\left(\frac{5}{4}\right)^{2n}\delta^{2}. 
\]
Denote $\kappa = M_{0}\left(\frac{5}{4}\right)^{2n} $. Then since $\mathbb{A}\in \mathcal{A}_{4}(\delta, \mu, \Lambda, B_{1})$ by assumption, the above calculation shows that conditions in  \eqref{cond-lemma} are satisfied for $u$ replace by $ u_{\kappa} = u/\kappa$ and ${\bf F}$ replaced by ${\bf F}_{\kappa} = {\bf F}/\kappa$,  where $u_{\kappa}$ will remain a weak solution corresponding to ${\bf F}/k$.  So all in \eqref{conclusion-lemma} will be true where $v$ will be replaced by  $v_{\kappa} :=v/\kappa$.  

\noindent {\it Second}, with $M^{2} = \max\{M_{0} 3^{2n}, 4C^{2}_{0} \}$, we then claim that  
\[
\{
x: \mathcal{M}^{\mu}(\chi_{B_{6}}|\nabla u_{\kappa}|) ^{2} > M^{2}
\}\cap B_{1} \subset \{x: \mathcal{M}^{\mu}(\chi_{B_{2}}|\nabla u_{\kappa} -\nabla v_{\kappa}|^{2})>C_{0}^{2}\}\cap B_{1}. \]
In fact,  otherwise there will exist  $x\in B_{1}$, such that
\[ \mathcal{M}^{\mu}(\chi_{B_{6}}|\nabla u_{\kappa}|) ^{2}(x) > M^2, \quad \text{and} \quad \mathcal{M}^{\mu}(\chi_{B_{2}}|\nabla u_{\kappa} -\nabla v_{\kappa}|^{2})(x)\leq C_{0}^{2}.
\]
We obtain a contradiction if we show that for any  $r>0$
\[
\fint_{B_{r}(x)}\chi_{B_{6}}|\nabla u_{\kappa}|^{2}d\mu\leq  M^{2}. 
\]
To that end, on the one hand, if $r \leq 1$, then $B_{r}(x)\subset B_{2}$. Using the fact that $\|\nabla v_{\kappa}\|_{L^{\infty}(B_{r}(x))} \leq \|\nabla v_{\kappa}\|_{L^{\infty}(B_{2})} \leq C_{0}$, 
\[
\fint_{B_{r}(x)}\chi_{B_{6}}|\nabla u_{\kappa}|^{2}d\mu\leq  2\fint_{B_{r}(x)}\chi_{B_{2}}|\nabla u_{\kappa}-\nabla v_{\kappa}|^{2} d\mu+ 2\fint_{B_{r}(x)}|\nabla v_{\kappa}|^{2}d\mu \leq 4 C_{0}^{2}. 
\]
On the other hand, if $r > 1$, then note first that $B_{r}(x) \subset B_{3r}(x_0)$ and, so scaling the first inequality in \eqref{level-set-eq1} by $\kappa > 1$ we obtain that 
\[
\fint_{B_{r}(x)}\chi_{B_{6}}|\nabla u_{\kappa}|^{2}d\mu(x) \leq \frac{\mu(B_{3r}(x_{0}))}{B_{r}(x)}\fint_{B_{3r}(x_0)} \chi_{B_{6}}|\nabla u_{\kappa}|^{2}d\mu(x) <M_{0}3^{2n}. 
\]
{\it Finally, } set $\varpi = \kappa M $. Then since $\varpi >M $ we have that 
\[
\begin{split}
\mu( \{ x\in B_{1}: \mathcal{M}^{\mu}(\chi_{B_{6}}|\nabla u|^{2}) > \varpi^{2} \})& \leq \mu(\{x\in B_{1}: \mathcal{M}^{\mu}(\chi_{B_{6}} |\nabla u_{\kappa}|^{2}) > M^{2} \})\\
& \leq \mu (\{x\in B_{1}: \mathcal{M}^{\mu}(\chi_{B_{2}}|\nabla u_{\kappa} -\nabla v_{\kappa}|^{2})>C_{0}^{2}\})\\
& \leq \frac{C(n, M_{0})}{C_{0}^{2}}\ \mu(B_{2}) \fint_{B_{2}} |\nabla u_{\kappa} - \nabla v_{\kappa}|^{2} d\mu \\
&\leq C\ \eta\ \mu(B_{2}) \\
&\leq C\  2^{2n} \ \eta\ \mu(B_{1}),
\end{split}
\]
where $C(n, M_{0})$ comes from the weak $1-1$ estimates in the $\mu$ measure and we have used \eqref{conclusion-lemma}. 
From the last estimate, we observe that if we choose $\eta > 0$ sufficiently small such that $ C \,2^{2n} \eta < \epsilon$, Lemma \ref{varpi-lemma} follows. 
\end{proof}
By scaling and translating, we can derive the following corollary from Lemma \ref{varpi-lemma}. 
\begin{corollary} \label{cor-varpi}
Suppose that $M_{0}>0$ and $\mu \in A_{2}$ such that $[\mu]_{A_{2}} \leq M_{0}$. 
There exists a constant $\varpi = \varpi(n, \Lambda, M_0)> 1$ such that the following holds true. Corresponding to any $\epsilon > 0 $, there exists a small constant $\delta= \delta(\epsilon, \Lambda, M_0, n)$ such that for any $\rho \in (0, 1), $ $y\in B_{1}$, 
$\mathbb{A}\in \mathcal{A}_{4\rho}(\delta, \mu, \Lambda, B_{\rho}(y))$,  if $u\in W^{1, 2}(B_{6}, \mu)$ is a weak solution to 
\[
\textup{div}[\mathbb{A} \nabla u] = \textup{div}({\bf F})\quad \text{in $B_{6}$}, \quad \text{and}
\]
\begin{equation}\label{max-f-small-foranyrho}
B_{\rho}(y)\cap \{x\in \mathbb{R}^{n}: \mathcal{M}^{\mu}(\chi_{B_{6}}|\nabla u|^{2}) \leq 1 \}\cap \{x\in \mathbb{R}^{n}: \mathcal{M}^{\mu}\left(\left|\frac{{\bf F}}{\mu}\right|^{2}\chi_{B_{6}}\right) \leq \delta^{2} \} \neq \emptyset, 
\end{equation}
then 
\[
\mu(\{x\in \mathbb{R}^{n}: \mathcal{M}^{\mu}(\chi_{B_{6}}|\nabla u|^{2})  > \varpi^{2}\}\cap B_{\rho}(y)) < \epsilon \mu(B_{\rho}(y)). 
\]
\end{corollary}
The following statement is the contrapositive of the above corollary.  
\begin{proposition}\label{contra-interior}
Suppose that $M_{0}>0$ and $\mu \in A_{2}$ such that $[\mu]_{A_{2}} \leq M_{0}$. 
There exists a constant $\varpi =\varpi(n, \Lambda, M_0)> 1$ such that the following holds true. Corresponding to any $\epsilon > 0$, there exists a small constant $\delta= \delta(\epsilon, \Lambda, M_0, n)$ such that for any $\rho \in (0, 1), $ $y\in B_{1}$, 
$\mathbb{A}\in \mathcal{A}_{4\rho}(\delta, \mu, \Lambda,  B_{\rho}(y))$,  if $u\in W^{1, 2}(B_{6}, \mu)$ is a weak solution to 
\[
\textup{div}[\mathbb{A} \nabla u] = \textup{div}({\bf F})\quad \text{in $B_{6}$}, \quad \text{and}
\]

\[
\mu( \{x\in \mathbb{R}^{n}: \mathcal{M}^{\mu}(\chi_{B_{6}}|\nabla u|^{2})  > \varpi^{2}\}\cap B_{\rho}(y)) \geq \epsilon \mu(B_{\rho}(y)), 
\]
then 
\begin{equation}\label{max-f-small-contra}
B_{\rho}(y)\subset  \{x\in B_{\rho}(y): \mathcal{M}^{\mu}(\chi_{B_{6}} |\nabla u|^{2}) > 1 \}\cup \{x\in B_{\rho}(y): \mathcal{M}^{\mu}\left(\left|\frac{{\bf F}}{\mu}\right|^{2} \chi_{B_{6}}\right) >\delta^{2} \}. 
\end{equation}
\end{proposition}
Our next statement, which is the key in obtaining the higher gradient integrability of solution, gives the level set estimate of $\mathcal{M}^{\mu}(\chi_{B_{6}}|\nabla u|^{2})$ in terms that of $\mathcal{M}^{\mu}\left(\left|\frac{{\bf F}}{\mu}\right|^{2} \chi_{B_{6}}\right)$.  
\begin{lemma} \label{interior-density-est-l}  Let $M_{0} >0$, $\mu \in A_{2}$ such that $[\mu]_{A_{2}} \leq M_{0}$, and let $\varpi$ be as in Proposition \ref{contra-interior}. Then, for every  $\epsilon > 0 $, there is $\delta= \delta(\epsilon, \Lambda, M_0, n) < 1/4$  such that the following holds:  For $\mathbb{A}\in \mathcal{A}_{4}(\delta, \mu, \Lambda, B_{2})$, ${\bf F}/\mu \in L^2(B_6, \mu)$, if $u\in W^{1, 2}(B_{6}, \mu)$ is a weak solution to 
\[
\textup{div}[\mathbb{A} \nabla u] = \textup{div}({\bf F})\quad \text{in} \quad B_{6},
\]
and  
\[
\mu(B_{1}\cap \{x\in \mathbb{R}^{n}: \mathcal{M}^{\mu}(\chi_{B_{6}}|\nabla u|^{2})  > \varpi^{2}\}) <\epsilon \mu(B_\frac{1}{2}(y)), \quad \forall \ y \in B_1,
\]
then for any $k\in \mathbb{N}$ and $\epsilon_{1} = \left(\frac{10}{1 - 4\delta}\right)^{2n} M_{0}^{2} \epsilon$ we have that 
\[
\begin{split}
 \mu(\{x\in B_{1}: \mathcal{M}^{\mu}(\chi_{B_{6}}|\nabla u|^{2}) > \varpi^{2k} \}) &\leq \sum_{i=1}^{k} \epsilon_{1}^{i} \mu\left(\{x\in B_{1}: \mathcal{M}^{\mu}\left(\left|\frac{{\bf F}}{\mu}\right|^{2} \chi_{B_{6}}\right) >\delta^{2} \varpi^{2(k-i)} \}\right)\\
&\quad\quad+ \epsilon_{1}^{k}\mu(\{x\in B_{1}: \mathcal{M}^{\mu}(\chi_{B_{6}} |\nabla u|^{2}) > 1 \}).
\end{split}
\]
\end{lemma}
 \begin{proof}
 We will use induction to prove the corollary. For the case $k=1$, we are going to apply Lemma \ref{Vitali}, by taking 
 \[C = \{x\in \mathbb{R}^{n}: \mathcal{M}^{\mu}(\chi{B_{6}}|\nabla u|^{2}) > \varpi^{2} \}\cap B_{1}
 \] 
and 
\[
D = \left(\{x\in  \mathbb{R}^{n}: \mathcal{M}^{\mu}\left(\left|\frac{{\bf F}}{\mu}\right|^{2} \chi_{B_{6}}\right) >\delta^{2}  \}\cup \{x\in  \mathbb{R}^{n}: \mathcal{M}^{\mu}(\chi_{B_{6}} |\nabla u|^{2}) > 1 \}\right)\cap  B_{1}. 
\]
Clearly, $C \subset D \subset B_1$. Moreover, by the assumption,  $\mu(C) < \epsilon  \mu(B_\frac{1}{2}(y))$, for all $y \in B_1$. Also for any $y\in B_{1}$ and $\rho\in (0,  1) $, then $\mathbb{A}\in \mathcal{A}_{4}(\delta, \mu, \Lambda, B_{2})$ implies that $\mathbb{A}\in \mathcal{A}_{4\rho}(\delta, \mu, \Lambda, B_{\rho}(y))$. Moreover, if $\mu (C \cap B_{\rho}(y))\geq \epsilon \mu((B_{\rho}(y))$, then   by Proposition \ref{contra-interior} we have that 
\[
B_{\rho}(y)\cap B_{1}\subset D.
\]
Hence,  all the conditions of Lemma \ref{Vitali} are satisfied and hence
\[
\mu(C) \leq \epsilon_{1} \mu(D).
\]
That proves the case when $k=1$. Assume it is true for $k$. We will show the statement for $k+1$. We normalize $u$ to $u_{\varpi} = u/\varpi$ and ${\bf F}_{\varpi} = {\bf F}/\varpi$, and we see that since $\varpi > 1$ we have 
\[
\begin{split}
\mu(\{x\in B_{1}: \mathcal{M}^{\mu}(\chi_{B_{6}}|\nabla u_{\varpi}|^{2}) > \varpi^{2} \}) &= \mu(\{x\in B_{1}: \mathcal{M}^{\mu}(\chi_{B_{6}}|\nabla u|^{2}) > \varpi^{4} \})\\
& \leq \mu(\{x\in B_{1}: \mathcal{M}^{\mu}(\chi_{B_{6}}|\nabla u|^{2}) > \varpi^{2} \}) \leq \epsilon \mu(B_\frac{1}{2}(y)), \quad \forall y \in B_1.
\end{split}
\]
By induction assumption, it follows then that 
\[
\begin{split}
 \mu(\{x\in B_{1}: \mathcal{M}^{\mu}(\chi_{B_{6}}|\nabla u|^{2}) > \varpi^{2(k+1)} \})&=
 \mu(\{x\in B_{1}: \mathcal{M}^{\mu}(\chi_{B_{6}}|\nabla u_{\varpi}|^{2}) > \varpi^{2k} \})\\
  &\leq \sum_{i=1}^{k} \epsilon_{1}^{i} \mu\left(\{x\in B_{1}: \mathcal{M}^{\mu}\left(\left|\frac{F_{\varpi}}{\mu}\right|^{2} \chi_{B_{6}}\right) >\delta^{2} \varpi^{2(k-i)} \}\right)\\
&\quad\quad+ \epsilon_{1}^{k}\mu(\{x\in B_{1}: \mathcal{M}^{\mu}(\chi_{B_{6}} |\nabla u_{\varpi}|^{2}) > 1 \})\\
& =  \sum_{i=1}^{k} \epsilon_{1}^{i} \mu\left(\{x\in B_{1}: \mathcal{M}^{\mu}\left(\left|\frac{{\bf F}}{\mu}\right|^{2} \chi_{B_{6}}\right) >\delta^{2} \varpi^{2(k+1-i)} \}\right)\\
&\quad\quad+ \epsilon_{1}^{k}\mu(\{x\in B_{1}: \mathcal{M}^{\mu}(\chi_{B_{6}} |\nabla u|^{2}) > \varpi^2 \}). 
\end{split}
\]
Applying the case $k=1$ to the last term we obtain that 
\[
\begin{split}
 \mu(\{x\in B_{1}: \mathcal{M}^{\mu}(\chi_{B_{6}}|\nabla u|^{2}) > \varpi^{2(k+1)} \})&\leq \sum_{i=1}^{k + 1} \epsilon_{1}^{i} \mu\left(\{x\in B_{1}: \mathcal{M}^{\mu}\left(\left|\frac{{\bf F}}{\mu}\right|^{2} \chi_{B_{6}}\right) >\delta^{2} \varpi^{2(k+1-i)} \}\right)\\
&\quad\quad+ \epsilon_{1}^{k+1}\mu(\{x\in B_{1}: \mathcal{M}^{\mu}(\chi_{B_{6}} |\nabla u|^{2}) > 1 \}), 
\end{split}
\]
as desired. 
 \end{proof}
\noindent
\begin{proof}[Completion of the proof of Theorem \ref{local-grad-estimate-interior}] With Lemma \ref{interior-density-est-l}, the rest of the proof of Theorem \ref{local-grad-estimate-interior} is now the same as {\bf Case I} in the proof of Theorem \ref{g-theorem} below. We therefore skip it.
\end{proof}
%===============
\section{Global $W^{1,p}$-regularity theory} \label{boundary-global-section}
\subsection{Boundary estimate setup}
We first introduce some notations. For $r >0$ and for $x_0~=~(x^{0}_1, x^{0}_2, \cdots, x^{0}_n)~\in~\mathbb{R}^n$, let us denote
\[
\begin{split}
B_r^+(x_0)  = \{ y = (y_1, y_2, \cdots, y_n) \in B_r(x_0):\ y_n >x^0_{n}\}, \quad B_r^+ = B_r^+(0), \\
T_r (x_0) = \{ x =  (x_1, x_2, \cdots, x_n) \in \partial B_{r}^+(0): \ x_n =x_{n}^{0}\}, \quad T_r = T_r(0).
\end{split}
\]
For $x_0 \in \mathbb{R}^n$, we also denote
\[
\Omega_r(x_0) = \Omega \cap B_r(x_0), \quad \partial_{w} \Omega_r(x_0) = \partial \Omega \cap B_r(x_0), \quad \Omega_r = \Omega_r(0).
\]
In this section we localize the problem near the boundary, assume that there is a coordinate system where for some $K>0$ and $\delta \in (0, 1/K)$ 
\[
B_{r}^{+} \subset \Omega_{r}\subset B_{r} \cap \{(y', y_{n} : y_{n} > -K\delta  r)\}, 
\]
and study the problems
\begin{equation} \label{boundary-eqn-r}
\left\{
\begin{array}{cccl}
\text{div}[\mathbb{A}(x) \nabla u]  & = & \text{div}[{\bf F}] & \quad \text{in} \quad \ \Omega_r, \\
u & =& 0 & \quad \text{on} \quad \partial_{w} \Omega_r, 
\end{array}
\right.
\end{equation}
and the corresponding homogeneous equation 
\begin{equation} \label{boundary-v-eqn-r}
\left\{
\begin{array}{cccl}
 \text{div}[\mathbb{A}_0 \nabla v ] &= &0 & \quad \text{in $B_r^+$,} \\
v &= & 0  & \quad \text{on $T_r$,}
\end{array}
\right.
\end{equation}
with a symmetric and elliptic constant matrix $\mathbb{A}_{0}$.  
\begin{definition}
\textup{(i)}  $u\in W^{1, 2}(\Omega_{r}, \mu)$ is a weak solution to \eqref{boundary-eqn-r} in $\Omega_{r}$ if  
\[
\int_{\Omega_{r}}\langle\mathbb{A}\nabla u, \nabla \varphi \rangle dx  = \int_{\Omega_{r}} \langle {\bf F},\nabla \varphi\rangle dx,\quad \forall \varphi\in C_{0}^{\infty}(\Omega_{r}), 
\]
and $u$'s zero extension to $B_r$ is in  $W^{1, 2}(B_{r}, \mu)$. 

\textup{(ii)} $v\in W^{1, q}(B_{r}^{+})$ is a weak solution to \eqref{boundary-v-eqn-r} in $B_{r}^{+}$,  for some $1 < q < \infty$, if 
\[
\int_{B_{r}^{+}} \langle\mathbb{A}_{0}\nabla v, \nabla \varphi \rangle dx  = 0,\quad \forall \varphi \in  C_{0}^{\infty}(B_{r}^{+}),
 \]
 and $v$'s zero extension to $B_r$ is also in $W^{1, q}(B_{r})$. 
\end{definition}
\noindent
Let us now consider the case when $r = 4$. The equation \eqref{boundary-v-eqn-r} in this case becomes
\begin{equation} \label{boundary-v-eqn}
\left\{
\begin{array}{cccl}
 \text{div}[\mathbb{A}_{0} \nabla v ] &= &0 & \quad \text{in}\, B_4^+, \\
v &= & 0  & \quad \text {in }\,\partial  T_4^.
\end{array}
\right.
\end{equation}
for some constant matrix $\mathbb{A}_{0}$ which will be chosen sufficiently close to $\langle \mathbb{A}\rangle_{B_4}$. Similar to Lemma \ref{reg-v}, the following boundary regularity result of the weak solution $v$ is also needed for our approximation estimates.
%=============
\begin{lemma} \label{boundary-reg-v} Let $\mathbb{A}_0$ be a constant symmetric matrix satisfying the ellipticity condition 
\[
\lambda_0 |\xi|^2 \leq \wei{\mathbb{A}_0\xi, \xi} \leq \Lambda_0 |\xi|^2, \quad \forall \ \xi \in \mathbb{R}^n,
\]
for some fixed positive constants $\Lambda_0, \lambda_0$. Then there exists a constant $C = C(n, \Lambda_0/\lambda_0)$ such that  if $v \in W^{1,q}(B_4^+)$ is a weak solution of \eqref{boundary-v-eqn} with some $q >1$, then 
\[
\norm{\nabla v}_{L^\infty(B_{\frac{7}{2}}^+)} \leq C \left( \fint_{B_{4}^+} |\nabla v|^{q} dx \right)^{\frac{1}{q}}.
\]
\end{lemma}
\begin{proof} The proof is the same as that of Lemma \ref{reg-v} but we use the boundary version \cite[Theorem A.2]{AMP} of \cite[Theorem A.1]{Brezis} instead when considering $1 < q <2$.
\end{proof}
%================
\subsection{Boundary weighted Caccioppoli estimate}
We assume that $\mathbb{A}$ is a measurable symmetric matrix, and there is $\Lambda>0$ such that  
 \begin{equation} \label{elip-boundary}
\Lambda \mu(x) |\xi|^2 \leq \wei{\mathbb{A}(x) \xi, \xi} \leq \Lambda^{-1} \mu(x)|\xi|^2 \quad \text{for a.e.\ } x\ \in \mathbb{R}^n, \quad \forall \ \xi \in \mathbb{R}^n.
\end{equation}
Let $u \in W^{1,2}(\Omega_4, \mu)$ be a weak solution of 
\begin{equation} \label{boundary-eqn}
\left\{
\begin{array}{cccl}
\text{div}[\mathbb{A}(x) \nabla u]  & = & \text{div}[{\bf F}]  & \quad \text{in} \quad \ \Omega_4, \\
u & =& 0 & \quad \text{on} \quad \partial_{w} \Omega_4.
\end{array}
\right.
\end{equation}
Let $v \in W^{1,1+\beta}(B_4^+)$ be a weak solution of \eqref{boundary-v-eqn}.  Similarly to  Lemma \ref{w-local-energy}, the following estimate is a weighted Caccioppoli estimate up to the boundary for the difference $u-v$.
%=============
\begin{lemma} \label{w-local-boundary-energy}  Assume that \eqref{elip-boundary} holds and $[\mu]_{A_2} \leq M_0$. Let  $\mathbb{A}_0 = (a_0^{ki})_{k,i=1}^n, v$ be as in Lemma \ref{boundary-reg-v}, and let $w = u- V$ where $V$ is the zero extension of $v$ in $B_4$.  There exists a constant $C = C(\Lambda, n,M_0)$ such that for all non-negative function $\varphi \in C_0^\infty(B_{4})$, 
\[
\begin{split}
\frac{1}{\mu(B_4)}\int_{\Omega_4} |\nabla w|^2 \varphi^2 d\mu & \leq C(\Lambda, M_0,n) \left [\frac{1}{\mu(B_4)} \int_{\Omega_{4}} w^2  |\varphi|^2 d\mu   +  \frac{1}{\mu(B_4)}\int_{\Omega_4} \Big | \frac{{\bf F}}{\mu} \Big |^2 \varphi^2 d\mu     \right.\\
&   \left. +\frac{ \norm{\varphi \nabla v}_{L^\infty(B_{4}^+)}^2}{\mu({B_4})} \int_{\Omega_4}| \mathbb{A} - \mathbb{A}_0|^2 \mu^{-1} dx \right.\\
&\left.+  |\A_0^{n,\cdot}|^2  \norm{ \varphi(x) \nabla v(x',0)}_{L^\infty (\Omega_4 \setminus B_4^+)}^2 \left(\frac{|B_{4}|}{\mu(B_{4})}\right)^{2} \left(\frac{|\Omega_4 \setminus B_4^+|}{|B_4|}\right)^{\varrho}\right], 
\end{split}
\]
where $\varrho$ is as in \eqref{Ainfinity} and depends only on $n$ and $M_0$, and ${\A}_{0}^{n, \cdot} = (a_{0}^{n1},\dots,a_{0}^{nn})$ is the $n^{th }$ row of $\mathbb{A}_{0}$.  
\end{lemma}
\begin{proof} 
Observe that by Lemma \ref{boundary-reg-v}, $v$ is smooth in $\overline{B}^+_r$ for $0 < r <4$. Therefore, a simple integration by parts shows that 
$V$ is a weak solution of
\[
\left\{
\begin{array}{ccll}
\text{div}[\mathbb{A}_0 \nabla V ] & = & \frac{\partial }{\partial x_n} g & \quad \text{in} \quad \Omega_4 \\
 V & =& 0 & \quad \text{on} \quad \partial_{w} \Omega_4,
\end{array}
\right.
\]
where 
\[
g (x', x_n)= \langle \mathbb{A}_{0}^{n, \cdot}, \nabla v (x', 0)\rangle  \chi_{\{x_n <0\}}(x',x_n). \]
Then, it follows that $w$ is a weak solution of
\[
\left\{
\begin{array}{ccll}
\text{div}[\mathbb{A} \nabla w] & = & \text{div}[{\bf F} - (\mathbb{A} - \mathbb{A}_0) \nabla V ] + \frac{\partial }{\partial x_n} g& \quad \text{in} \quad \Omega_4, \\
w & =& 0 & \quad \text{on} \quad \partial_{w} \Omega_4.
\end{array} \right.
\]
For any non-negative cut-off function $\varphi \in C_0^\infty(B_4)$, it follows that $w\varphi^{2} \in W^{1,2}_{0}(\Omega_{4}, \mu)$. Therefore, by using $w\varphi^2$ as a test function for the above equation, we obtain 
\begin{equation}  \label{boundary- w-energy-test}
\begin{split}
 \int_{\Omega_4} \wei{\mathbb{A}\nabla w, \nabla w}\varphi^2 dx  &= - \int_{\Omega_4} \wei{\mathbb{A}\nabla w, \nabla (\varphi^2)} w dx  +\int_{\Omega_4} \wei{{\bf F}, \nabla (w\varphi^2)} dx \\
 & \quad \quad - \int_{\Omega_4} \wei{(\mathbb{A}- \mathbb{A}_0) \nabla V, \nabla(w\varphi^2)}dx +\int_{\Omega_{4}}  g\frac{\partial }{\partial x_n}  [w\varphi^2] dx. 
\end{split}
\end{equation}
 Except the last term on the right hand side of \eqref{boundary- w-energy-test}, all terms in \eqref{boundary- w-energy-test} can be estimated exactly as in those in Lemma \ref{w-local-energy}. To estimate this last term, we use Young's inequality to write 
 \[
\begin{split}
\left|  \int_{\Omega_{4}} g \frac{\partial}{\partial x_n} [w\varphi^2] dx \right| &\leq\int_{\Omega_{4}} |g| \Big[| \nabla w| \varphi^2 + 2|\nabla \varphi| \varphi |w|\Big ] \\
&\leq\epsilon\int_{\Omega_{4}} |\nabla w|^{2}\varphi^{2}\mu(x)dx + \epsilon \int_{\Omega_{4}} |w|^{2}|\nabla \varphi|^{2}\mu(x)dx  + C_{\epsilon} \int_{\Omega_{4}}|g|^{2}\varphi^2 \mu^{-1} dx.  \\
\end{split}
\]
Next, we use the doubling property of weights and the fact that $g$ is supported on $\Omega_{4}\setminus B_{4}^{+}$ to have 
\[
\begin{split}
& C_{\epsilon} \int_{\Omega_{4}}|g|^{2}\varphi^2 \mu^{-1} dx \leq \|g\varphi\|^{2}_{L^{\infty}(\Omega_4 \setminus B_4^+)}\mu^{-1}(\Omega_4 \setminus B_4^+) \\
&\leq C(M_{0}, \epsilon, n)\norm{g\varphi}^{2}_{L^{\infty}(\Omega_4 \setminus B_4^+)} \left(\frac{|\Omega_4 \setminus B_4^+|}{|B_4|}\right)^{\varrho} \frac{|B_4|^{2}}{\mu(B_{4})}, 
\end{split}
\]
for some constant $\varrho$ is as in \eqref{Ainfinity} and that depends only on $n$ and $M_{0}$.  
Then, we can follow the proof of Lemma \ref{w-local-energy} to derive the estimate in Lemma \ref{w-local-boundary-energy}.
\end{proof}
\subsection{Boundary gradient approximation estimates}
We begin the section with the following lemma showing the $u$ can be approximated by $v$ in $L^2(\Omega_{7/2}, \mu)$.
%======
\begin{lemma} \label{L2-boundary-u-v-aprox} Let $K>0, \Lambda >0, M_0 >0$ be fixed and let $\beta$ be as in \eqref{l-s.def}. For every $\epsilon >0$ sufficiently small, there exists $\delta\in (0, 1/K)$ depending on only $\epsilon, \Lambda, n$, and $M_0$ such that the following statement holds true: If $\mathbb{A}, \mu, {\bf F}$ such that $[\mu]_{A_2} \leq M_0$, \eqref{elip-boundary} holds, 
\[
  B_4^+ \subset  \Omega_4 \subset B_4 \cap \{ x_n > -4K\delta \}, 
\]
and 
\[ 
\begin{split}
&  \frac{1}{\mu({B_4})} \int_{\Omega_4}| \mathbb{A} - \langle\mathbb{A}\rangle_{B_{4}}|^2 \mu^{-1} dx + \frac{1}{\mu(B_4)}\int_{\Omega_4} \Big | \frac{{\bf F}}{\mu} \Big |^2 d\mu (x) \leq \delta^2, 
\end{split}
\]
a weak solution $u \in W^{1,2}(\Omega_4, \mu)$ of \eqref{boundary-eqn} that satisfies
\begin{equation} \label{L2-gradient-u}
  \frac{1}{\mu(B_4) }\int_{\Omega_4} |\nabla u|^2 d\mu \leq 1,
\end{equation}
then 
there exists a constant matrix $\mathbb{A}_0$ and a weak solution $v \in W^{1,1+\beta}(B_4^+)$ of \eqref{boundary-v-eqn} such that
\[
\norm{\langle \mathbb{A}\rangle_{B_4} - \mathbb{A}_0} \leq \frac{\epsilon  \mu(B_{4})}{|B_{4}|}, 
\]
and
\[
\frac{1}{\mu(B_{7/2})}\int_{\Omega_{7/2}} |u  - V|^2 d\mu \leq \epsilon,
\]
where $V$ is the zero extension of $v$ to $B_4$. Moreover, there is $C =C(\Lambda, n, M_0)$ such that
\begin{equation} \label{b-L2-gradient-v}
\fint_{B_{3}^+}|\nabla v|^2dx \leq C(\Lambda, n, M_0).
\end{equation}
\end{lemma} 
\begin{proof}  As in the proof of Lemma \ref{L2-aprox}, we can use appropriate scaling to assume that
\[
\wei{\mu}_{B_4} = \frac{1}{|B_4|}\int_{B_4} \mu(x) dx = 1.
\]
We again proceed the proof with a contradiction argument. Suppose that there exists $\epsilon_{0} > 0$ such that  for every $k \in \mathbb{N}$, there are $\mu_k \in A_2$, $ \mathbb{A}_k$  satisfying the degenerate ellipticity assumption as in \eqref{elip-boundary} with $\mu$ and $\mathbb{A}$ are replaced by $\mu_k$ and $\mathbb{A}_k$ respectively,  and domain $\Omega_{4}^{k}$, ${\bf F}_k$ and  a weak solution $u_k \in W^{1, 2}(\Omega_{4}^{k}, \mu_{k})$ of 
\begin{equation} \label{b-u-k.eqn} \left\{
\begin{array}{cccl}
 \text{div}[\mathbb{A}_k \nabla u_k] & = & \text{div}({\bf F}_k) & \quad  \text{in} \quad \Omega^{k}_4, \\
  u & =& 0 & \quad \text{on} \quad \partial_w \Omega^{k}_4, 
  \end{array} \right.
\end{equation}
with 
\begin{equation} \label{B-flat-k}
B_4^+ \subset \Omega_4^k \subset B_4 \cap \Big \{x_n \geq -\frac{4 K}{k} \Big \},
\end{equation}
\begin{equation} \label{b-a_k} \left\{
\begin{split} 
& \frac{1}{\mu_{k}({B_4})} \int_{\Omega^{k}_4}| \mathbb{A}_{k} - \langle\mathbb{A}_{k}\rangle_{B_{4}}|^2 \mu_{k}^{-1} dx + \frac{1}{\mu_{k}(B_4)}\int_{\Omega^{k}_4} \Big | \frac{{\bf F}_{k}}{\mu_{k}} \Big |^2 d\mu_{k} (x) \leq \frac{1}{k^2}, \\
& [\mu_k]_{A_2} \leq M_0, \quad \wei{\mu}_{k, B_4} = \frac{1}{|B_4|} \int_{B_4} \mu_k(x) dx = 1
\end{split} \right.
\end{equation}
and 
\begin{equation} \label{b-gradient-b-k}
\frac{1}{\mu_{k}(B_4)}  \int_{\Omega_4^k} |\nabla u_k|^2 d\mu_k \leq 1,
\end{equation}
but for all elliptic, symmetric and constant matrix $\mathbb{A}_0$ with
\[
 \norm{\langle\mathbb{A}_{k}\rangle_{B_{4}} - \mathbb{A}_0} \leq \epsilon_{0},
\]
 and all weak solution $v \in W^{1, 1+\beta}(B_4^+)$ of \eqref{boundary-v-eqn} and $V$ is its zero extension to $B_4$, we have 
\begin{equation} \label{b-epsilon-0}
\frac{1}{\mu_k(B_{7/2})}\int_{\Omega^k_{7/2}} |u_{k} - V|^2 d\mu_k \geq \epsilon_0.
\end{equation}
It follows from \eqref{elip-boundary}  that
\begin{equation} \label{b-bar-a-ellip}
\Lambda  |\xi|^2 \leq \wei{\langle\mathbb{A}_{k}\rangle_{B_{4}}\xi, \xi} \leq \Lambda ^{-1} |\xi|^2, \quad \forall \ \xi \in \mathbb{R}^n.
\end{equation}
Then, since the sequence $\{\langle\mathbb{A}_{k}\rangle_{B_{4}}\}_k$ is a bounded sequence  in $\mathbb{R}^{n\times n}$, by passing through a subsequence, we can assume that there is a constant matrix $\bar{\mathbb{A}}$ in $ \mathbb{R}^{n\times n}$ such that 
\begin{equation} \label{a-converge}
\lim_{k \rightarrow \infty}\langle\mathbb{A}_{k}\rangle_{B_{4}} = \bar{\mathbb{A}}
\end{equation}
From \eqref{b-gradient-b-k}, and  Poincar\'e-Sobolev inequality \cite[Theorem 1.6]{Fabes}, we see that 
\[
\frac{1}{\mu_k(B_4)}\int_{\Omega_4^k} |u_k|^2 d \mu_k  \leq  \frac{C(n, M_{0})}{\mu_k(B_4)} \int_{\Omega_{4}^k} |\nabla u_k|^2 d\mu_k  \leq C(n, M_{0}),  \forall \ k \in \mathbb{N}. \]
This and since $\mu_k(B_4) = |B_4|$ for all $k$, it follows that 
\[
\norm{u_k}_{W^{1,2}(\Omega_4^k, \mu_k)} = \norm{u_k}_{L^2(\Omega_4^k, \mu_k)} + \norm{\nabla u_k}_{L^2(\Omega_4^k, \mu_k)} \leq C(n, M_0), \quad \forall k \in \mathbb{N}.
\]
As a consequence, Lemma \ref{L-p-L2mu} implies that
\[
\norm{u_k}_{W^{1,1 +\beta}(B_4)} \leq C(n, M_0) \norm{u_k}_{W^{1,2}(\Omega_4^k, \mu_k)} \leq C(n, M_0), \quad \beta = \frac{\gamma}{2 + \gamma} >0.
\]
Here, note that in the above estimate we still denote $u_k \in W^{1, 1+\beta}(B_4)$ to be zero extension of $u_k$ to $B_4$. Also, recall that $\gamma$ is defined in Lemma \ref{A-2-reverse-Holder}, which only depends on $n$ and $M_0$. Therefore, by the compact imbedding $W^{1,1+\beta}(B_4)  \hookrightarrow L^{1+\beta}(B_4)$ and by passing through a subsequence, we can assume that there is $u \in W^{1,1+\beta}(B_4)$ such that
\begin{equation} \label{b-w-k-converge}
\left\{
\begin{split}
& u_k \to u  \mbox{ strongly in } L^{1+\beta}(B_4),\quad \nabla u_k \rightharpoonup \nabla u \text{ weakly in } L^{1+\beta}(B_4),\
 \quad \text{and} \\
& \hat{u}_{k} \rightarrow u \ \text{a.e. in} \  \ B_4.
\end{split} \right.
\end{equation}
Moreover,
\begin{equation} \label{b-u-bound}
\norm{u}_{W^{1,1+\beta}(B_4)} \leq C(n, M_0).
\end{equation}
We claim that $u \in W^{1,1+\beta}(B_4^+)$ is a weak solution of 
\begin{equation} \label{w-unbounded-lambda}
\left\{
\begin{array}{cccl}
 \text{div}[\bar{\mathbb{A}}\nabla u] & = & 0 &  \quad \text{in} \quad  B_4^+, \\
  u & =& 0 & \quad \text{on} \quad T_4.
  \end{array} \right.
\end{equation}
To prove this claim, first note that from \eqref{b-u-k.eqn}, \eqref{B-flat-k} and  \eqref{b-w-k-converge}, it follows that
\begin{equation} \label{u-zero-B4-}
u = 0 \quad \text{for a.e.} \quad x \in B_4 \setminus B_4^+.
\end{equation}
In particular, $u = 0$ on $T_4$ in trace sense. It is therefore enough to show that the weak form of the PDE in \eqref{w-unbounded-lambda} holds. To proceed, let us fix $\varphi \in C^\infty_0(B_4^+)$. Then, by using $\varphi $ as a test function for the equation \eqref{b-u-k.eqn} of $u_k$, we have
\begin{equation} \label{test-uk-inter}
\int_{B_4^+} \wei{\A_k \nabla u_k, \nabla \varphi} dx = \int_{B^+_4} \langle{\bf F}_k, \nabla \varphi \rangle dx.
\end{equation}
We will take the limit $k \rightarrow \infty$ on both sides of the above equation. First of all, observe that by H\"{o}lder's inequality and \eqref{b-a_k}, it follows that the right hand side term of \eqref{test-uk-inter} can be estimated as 
\[
\begin{split}
\left|\fint_{B^+_4} \langle {\bf F}_k, \nabla \varphi \rangle dx  \right| & \leq \left\{\fint_{B^+_4} \Big |\frac{{\bf F}_{k}}{\mu_k}\Big |^2 \mu_k dx \right\}^{1/2} \left\{\fint_{B^+_4} |\nabla \varphi|^2 \mu_k dx\right\}^{1/2} \\
& \leq \norm{\nabla \varphi}_{L^\infty(B^+_4)} \left\{\frac{1}{\mu_k (B_4)}\int_{\Omega^{k}_4}\Big |\frac{{\bf F}_{k}}{\mu_k} \Big |^2d \mu_k (x) \right\}^{1/2} \frac{\mu _{k}(B_{4})}{|B_{4}^{+}|}\\ 
&\leq 2\norm{\nabla \varphi}_{L^\infty(B^+_4)} \left\{\frac{1}{\mu_k (B_4)}\int_{\Omega^{k}_4}\Big |\frac{{\bf F}_{k}}{\mu_k} \Big |^2d \mu_k (x) \right\}^{1/2}\\
&  \leq 2 \frac{\norm{\nabla \varphi}_{L^\infty(B_4^+)}}{k} . 
\end{split}
\]
Therefore, taking the limit as $k \to \infty,$ we have
\begin{equation} \label{b-lim-1}
\int_{B_4^+} \langle {\bf F}_k, \nabla \varphi  \rangle dx \to 0. 
\end{equation}
On the other hand, it follows from  \eqref{b-a_k}, \eqref{b-gradient-b-k}, and H\"older's inequality  that 
\[
\begin{split}
 \left| \fint_{B_4^+} \wei{(\mathbb{A}_{k} - \langle\mathbb{A}_{k})\rangle_{B_{4}}\nabla u_k, \nabla \varphi} dx \right| & \leq \fint_{B_4^+} |\mathbb{A}_{k} - \langle\mathbb{A}_{k}\rangle_{B_{4}}| \mu_{k}^{-1/2}|\nabla u_k| \mu_k^{1/2} |\nabla \varphi|  dx \\
& \leq \|\nabla \varphi\|_{L^{\infty}(B_{4}^{+})}\left\{\fint_{B_4^+} |\mathbb{A}_{k} - \langle\mathbb{A}_{k}\rangle_{B_{4}}|^{2} \mu_{k}^{-1} dx\right\}^{1/2} \left\{\frac{1}{|B_4^+|}\int_{B_4^+} |\nabla u_k|^2 d\mu_k \right\}^{1/2} \\
&\leq 2\frac{\norm{\nabla \varphi}_{L^\infty(B_4^+)}}{k}  \left\{\frac{1}{\mu_k(B_4)}\int_{\Omega^{k}_4} |\nabla u_k|^2 d\mu_k \right\}^{1/2}  \\
& \leq 2\frac{\norm{\nabla \varphi}_{L^\infty(B_4^+)}}{k}   \rightarrow 0, \quad \text{as} \quad k \rightarrow \infty.
\end{split}
\]
As a result we have, 
\[
0 = \lim_{k\to \infty }  \int_{B_4^+} \wei{(\mathbb{A}_{k} - \langle\mathbb{A}_{k}\rangle_{B_{4}}) \nabla u_k, \nabla \varphi} dx =  \lim_{k\to \infty } \Big[\int_{B_4^+} \wei{\mathbb{A}_{k}  \nabla u_k, \nabla \varphi} dx - \int_{B_4^+} \wei{ \langle\mathbb{A}_{k}\rangle_{B_{4}} \nabla u_k, \nabla \varphi} dx\Big]. 
\]
We also observe that since $\nabla u_k$ converges weakly in $L^{1 + \beta}(B_4^+)$ from \eqref{b-w-k-converge} and the constant symmetric matrix $\langle\mathbb{A}_{k}\rangle_{B_{4}}$ converges to $\bar{\mathbb{A}}$, we have that 
\[
\lim_{k\rightarrow \infty} \int_{B_4^+} \wei{\langle\mathbb{A}_{k}\rangle_{B_{4}} \nabla u_k, \nabla \varphi}dx =  \int_{B_{4}^+}\wei{\bar{\mathbb{A}} \nabla u, \nabla\varphi} dx. 
\]
Combining the above we conclude that,
\[
 \int_{B_4^+} \wei{\bar{\mathbb{A}}\nabla u, \nabla \varphi} dx =0, \quad \forall \ \varphi \in C^\infty_0({B}_4^+).
\]
Now, from \eqref{b-bar-a-ellip}, and since $\bar{\mathbb{A}} = \lim_{k\rightarrow \infty} \langle\mathbb{A}_{k}\rangle_{B_{4}}$, we observe that
\[ 
 \Lambda |\xi|^2 \leq \wei{\bar{\mathbb{A}}\xi, \xi} \leq  \Lambda^{-1} |\xi|^2, \quad \forall \ \xi \in \mathbb{R}^n.
\]
In other words, the symmetric constant matrix $\bar{\mathbb{A}}$ is uniformly elliptic. Hence, Lemma \ref{reg-v} implies that $u~\in~W^{1, \infty}(\overline{B}_{15/4}^+)$. In addition, it follows from Lemma \ref{boundary-reg-v}, \eqref{b-a_k},  \eqref{b-u-bound}, and \eqref{u-zero-B4-} that
\begin{equation} \label{b-bound-u-mu-k}
\fint_{B_{7/2}} |\nabla u|^2 d\mu_k \leq \norm{\nabla u}_{L^\infty(B_{7/2}^+)}^2 \leq C(n, \Lambda) \left( \fint_{B_4^+} |\nabla u|^{1+\beta} dx  \right)^{\frac{2}{1+\delta}}\leq C(n, M_0, \Lambda)  \quad \forall \ k \in \mathbb{N}.
\end{equation}
As in the interior case, we claim that 
\[
\lim_{k\to \infty} \frac{1}{\mu_{k}(B_{7/2})} \int_{\Omega^k_{7/2}}| u_k - u|^{2} d\mu_{k} = 0.
\] 
The proof of this claim follows exactly as in that of Lemma \ref{L2-aprox}, where use the Poincar\'e-Sobolev inequality, \cite[Theorem 1.6]{Fabes} instead of \cite[Theorem 1.5]{Fabes} and after noting that \[
\frac{1}{\mu_{k}(B_{7/2})} \int_{\Omega^k_{7/2}}| u_k - u|^{2} d\mu_{k} = \fint_{B_{7/2}}| u_k - u|^{2} d\mu_{k}, 
\]
since both $u_{k}$ and $u$ vanish in $B_{7/2}\setminus\Omega_{7/2}^{k}$. 
Now, since $\langle\mathbb{A}_{k}\rangle_{B_{4}} \rightarrow \bar{\mathbb{A}}$ and $\epsilon_0 >0$
\[
\norm{\langle\mathbb{A}_{k}\rangle_{B_{4}} - \bar{\mathbb{A}}} \leq \epsilon_0, 
\]
for sufficiently large $k$. But, this and our last claim, contradict \eqref{b-epsilon-0} if we take $\mathbb{A}_0 = \bar{\mathbb{A}}$, $v = u$ and $k$ sufficiently large. \\

Finally, estimate \eqref{b-L2-gradient-v} can be established  exactly the same way  as \eqref{L2-gradient-v} using Lemma \ref{boundary-reg-v}. The proof of Lemma \ref{L2-boundary-u-v-aprox} is now complete.
\end{proof}
Next, using the energy estimates in Lemma \ref{w-local-boundary-energy} and Lemma \ref{L2-boundary-u-v-aprox}, we can prove the following result, which is also the main result of the section.
%==========================
\begin{proposition} \label{L2-boundary-gradient-aprox} Let $K>0, \Lambda >0, M_0 >0$ be fixed and let $\beta$ be as in \eqref{l-s.def}. For every $\epsilon >0$ sufficiently small, there exists $\delta \in (0, 1/K)$ depending on only $\epsilon, \Lambda, n$, and $M_0$ such that the following statement holds true: If $\mathbb{A}, \mu, {\bf F}$ such that $[\mu]_{A_2} \leq M_0$, \eqref{elip-boundary} holds, 
\[
\begin{split}
  & B_4^+ \subset  \Omega_4 \subset B_4 \cap \{ x_n > -4K\delta \}, \\
&  \frac{1}{\mu({B_4})} \int_{\Omega_4}| \mathbb{A} - \langle\mathbb{A}\rangle_{B_{4}}|^2 \mu^{-1} dx + \frac{1}{\mu(B_4)}\int_{\Omega_4} \Big | \frac{{\bf F}}{\mu} \Big |^2 d\mu (x)  \leq \delta^2, 
\end{split}
\]
a weak solution $u \in W^{1,2}(\Omega_4, \mu)$ of \eqref{boundary-eqn} that satisfies
\begin{equation} \label{L2-gradient-u-grad}
  \frac{1}{\mu(B_4) }\int_{\Omega_4} |\nabla u|^2 d\mu \leq 1,
\end{equation}
then 
there exists a constant matrix $\mathbb{A}_0$ and a weak solution $v \in W^{1,1+\beta}(B_4^+)$ of \eqref{boundary-v-eqn} such that
\begin{equation}\label{a_0-bar-a-4}
\norm{\langle \mathbb{A}\rangle_{B_4} - \mathbb{A}_0} \leq \frac{\epsilon  \mu(B_{4})}{|B_{4}|}, 
\quad\text{and}\quad 
\fint_{\Omega_2} |\nabla u  -\nabla V|^2 d\mu \leq \epsilon,
\end{equation}
where $V$ is the zero extension of $v$ to $B_4$. Moreover, there is $C =C(\Lambda, n, M_0)$ such that
\begin{equation*} 
\fint_{B_{3}^+}|\nabla v|^2dx \leq C(\Lambda, n, M_0).
\end{equation*}
\end{proposition}
%=================================
\begin{proof} 
It follows from  ellipticity  condition \eqref{elip-boundary} of  $\mathbb{A}$ that 
\[
 \Lambda\frac{\mu( B_4)}{|B_4|}|\xi|^2 \leq \wei{\langle\mathbb{A} \rangle_{B_{4}}\xi, \xi}   \leq \Lambda^{-1}\frac{\mu(B_4)}{|B_4|}|\xi|^2, \quad \forall \ \xi \in \mathbb{R}^n.
\]
Therefore, if $\epsilon$ is sufficiently small, and $\mathbb{A}_{0}$ such that \eqref{a_0-bar-a-4} holds, it follows that
\[
\Lambda \frac{ \mu(B_4)}{2 |B_4|} |\xi|^2\leq \wei{\mathbb{A}_{0} \xi, \xi}  \leq \Lambda^{-1} \frac{2 \mu(B_4)}{|B_4|} |\xi|^2, \quad \forall \ \xi \in \mathbb{R}^n.
\]
This particularly implies
\begin{equation} \label{a-0-nn-bound}
|a_0^{ni}|\frac{|B_{4}|}{\mu(B_4)} \leq 2\Lambda^{-1}, \quad \text{for $i = 1,\dots,n$}.  
\end{equation}
Moreover, observe that by the flatness assumption
\[
\frac{|\Omega_4\setminus B_4^+|}{|B_4|} \leq C(n) K\delta.
\]
 The remaining part of the proof can be done similarly to that of Proposition \ref{L2-gradient-aprox} using  the last estimate, \eqref{a-0-nn-bound}, and Lemmas \ref{w-local-boundary-energy} -\ref{L2-boundary-u-v-aprox}.
\end{proof}
\subsection{Level set estimates up to the boundary} \label{b-density-est}
We begin with the following result on the density of interior level sets which is a consequence of Proposition \ref{contra-interior}.
\begin{proposition}\label{theorem-contra-interior}
Suppose that $M_{0}>0$ and $\mu \in A_{2}$ such that $[\mu]_{A_{2}} \leq M_{0}$. 
There exists a constant $\varpi = \varpi(n, \Lambda, M_0)> 1$ such that the following holds true. Corresponding to any $\epsilon > 0 $ there exists a small constant $\delta= \delta(\epsilon)$ such that for any $r > 0,$ $y\in \Omega$ such that $B_{6r}(y)\subset \Omega$, 
$\mathbb{A}\in \mathcal{A}_{6r}(\delta, \mu, \Lambda, \Omega)$,  if $u\in W^{1, 2}(\Omega, \mu)$ is a weak solution to 
\[
\textup{div}[\mathbb{A} \nabla u] = \textup{div}({\bf F})\quad \text{in $\Omega$}, \quad \text{and}
\]

\[
\mu( \{x\in \mathbb{R}^{n}: \mathcal{M}^{\mu}(\chi_{\Omega}|\nabla u|^{2})  > \varpi^{2}\}\cap B_{r}(y)) \geq \epsilon \mu(B_{r}(y)), 
\]
then 
\begin{equation}\label{max-f-small-contra-theorem}
B_{r}(y)\subset  \{x\in B_{r}(y): \mathcal{M}^{\mu}(\chi_{\Omega} |\nabla u|^{2}) > 1 \}\cup \{x\in B_{r}(y): \mathcal{M}^{\mu}\left(\left|\frac{{\bf F}}{\mu}\right|^{2} \chi_{\Omega}\right) >\delta^{2} \}. 
\end{equation}
\end{proposition}
Our next goal is to obtain similar result as Proposition \ref{theorem-contra-interior} but for balls that may intersect the boundary of the domain $\Omega$.   We begin with the following local near boundary estimate.  

\begin{lemma} \label{level-Omega-4}
Suppose that $M_{0}>0$ and $\mu \in A_{2}$ such that $[\mu]_{A_{2}} \leq M_{0}$. 
There exists a constant $\varpi(n, \Lambda, M_0)> 1$ such that the following holds true. Corresponding to any $\epsilon > 0 $, there exists a small constant $\delta~=~\delta(\epsilon, \Lambda, M_0, n)$ such that
$\mathbb{A}\in \mathcal{A}_{4}(\delta, \mu, \Lambda,  B_{1})$,  $u\in W^{1, 2}_{0}(\Omega, \mu)$ is a weak solution to \eqref{main-eqn}
where 
\begin{equation}\label{RF6}
B_{6}^{+}\subset \Omega_{6}\subset B_{6}\cap \{x_{n} > -12 \delta\}, 
\end{equation}
\begin{equation}\label{max-f-small-forrho1-B}
\Omega_{1}\cap \{x\in \mathbb{R}^{n}: \mathcal{M}^{\mu}(\chi_{\Omega}|\nabla u|^{2}) \leq 1 \}\cap \{x\in \mathbb{R}^{n}: \mathcal{M}^{\mu}\left(\left|\frac{{\bf F}}{\mu}\right|^{2}\chi_{\Omega}\right) \leq \delta^{2} \} \neq \emptyset, 
\end{equation}
then 
\[
\mu(\{x\in \Omega_{1}: \mathcal{M}^{\mu}(\chi_{\Omega}|\nabla u|^{2})  > \varpi^{2}\}) < \epsilon \mu(B_{1}). 
\]
\end{lemma}
\begin{proof}
Let $\epsilon > 0$ be given. The assumption \eqref{RF6} implies  that 
\[
B_{4}^{+}\subset \Omega_{4}\subset B_{4}\cap \{x_{n} > -16 \delta\}. 
\]
By  Lemma \ref{boundary-reg-v} and Proposition \ref{L2-boundary-gradient-aprox} corresponding to $K=4$
 the following holds: 
for every $\eta > 0$ there exits $\delta \in (0, 1/4)$ such that if $u$ is a weak solution to \eqref{boundary-eqn},  
\begin{equation}\label{cond-lemma-B}
\frac{1}{\mu(B_{4})} \int_{\Omega_{4}} |\mathbb{A} - \langle\mathbb{A}\rangle_{B_{4}}|^{2} \mu^{-1}dx  \leq \delta, \quad \frac{1}{\mu(B_{4})}\int_{\Omega_{4}} \Big | \frac{{\bf F}}{\mu}\Big |^2 d\mu(x) \leq \delta^2, \text{and}\quad \frac{1}{\mu(B_{4})}\int_{\Omega_{4}} |\nabla u|^{2}d\mu \leq 1
\end{equation}
then there exists a constant matrix $\mathbb{A}_{0}$ and a weak solution $v$ to \eqref{boundary-v-eqn} 
satisfying 
\begin{equation}\label{conclusion-lemma-B}
\|\langle\mathbb{A}\rangle_{B_{4}} - \mathbb{A}_{0}\| < \eta\, \frac{\mu(B_{4})}{|B_{4}|},\,\, \frac{1}{\mu(B_{2})}\int_{\Omega_{2}}|\nabla u-\nabla v| ^{2} d\mu< \eta, 
\quad \text{and\, $\| \nabla v\| _{L^{\infty}(B_{2}^+)} \leq C_{0}$},
\end{equation}
 for some positive constant $C_{0}$ that depend only $n, \Lambda$ and $M_{0}$.
Using this $\delta$ 
in the assumption \eqref{max-f-small-forrho1-B} 
there exists $x_{0}\in \Omega_{1}$ such that for any $r > 0$ 
\begin{equation}\label{level-set-eq1-B}
\frac{1}{\mu(B_{r}(x_{0}))}\int_{\Omega_{r}(x_{0})} \chi_{\Omega}|\nabla u|^{2} d\mu(x) \leq 1,\quad 
\text{and } \frac{1}{\mu(B_{r}(x_{0}))}\int_{\Omega_{r}(x_{0})} \left|\frac{{\bf F}}{\mu}\right|^{2} \chi_{\Omega} d\mu(x) \leq \delta^{2}. 
\end{equation}
We now  make several  observations. 

\noindent{\it First,}  $\Omega_{4}(0)\subset \Omega_{5}(x_{0})\subset \Omega$ and $B_{4}^{+}\subset \Omega_{4}\subset B_{4}\cap\{x_{n} > -16\delta \}$ and 
therefore we have from \eqref{level-set-eq1-B} that 
\[
\frac{1}{\mu(B_{4})}\int_{\Omega_{4}} |\nabla u|^{2} d\mu(x)\leq \frac{\mu(B_{5}(x_0))}{\mu(B_{4})} \frac{1}{\mu(B_{5}(x_{0}))} \int_{\Omega_{5}(x_0)} \chi_{\Omega}|\nabla u|^{2} d\mu(x) \leq M_{0}\left(\frac{5}{4}\right)^{2n}, 
\]
and similarly 
\[
 \frac{1}{\mu(B_{4})}\int_{\Omega_{4}} \left|\frac{{\bf F}}{\mu}\right|^{2}d\mu(x) \leq M_{0}\left(\frac{5}{4}\right)^{2n}\delta^{2}. 
\]
Denote $\kappa = M_{0}\left(\frac{5}{4}\right)^{2n}$. Then since $\mathbb{A}\in \mathcal{A}_{6}(\delta, \mu, \Lambda, B_{1})$ by assumption, the above calculation show that conditions in  \eqref{cond-lemma-B} are satisfied for $u$ replace by $ u_{\kappa} = u/\kappa$ and ${\bf F}$ replaced by ${\bf F} _{\kappa} = {\bf F}/\kappa$,  where $u_{\kappa}$ will remain a weak solution corresponding to ${\bf F}/k$.  So all in \eqref{conclusion-lemma-B} will be true where $v$ will be replaced by  $v_{\kappa} := v/\kappa$.  \\ \ \\
{\it Second}, taking $M^{2} = \max\{M_{0} 3^{2n}, 4C^{2}_{0} \}$, we have that  
\[
\{
x: \mathcal{M}^{\mu}(\chi_{\Omega}|\nabla u_{\kappa}|) ^{2} > M^{2}
\}\cap \Omega_{1} \subset \{x: \mathcal{M}^{\mu}(\chi_{\Omega_{2}}|\nabla u_{\kappa} -\nabla v_{\kappa}|^{2})>C_{0}^{2}\}\cap \Omega_{1}. \]
In fact,  otherwise there will exist  $x\in \Omega_{1}$, and $ \mathcal{M}^{\mu}(\chi_{\Omega_{2}}|\nabla u_{\kappa} -\nabla v_{\kappa}|^{2})(x)\leq C_{0}^{2}$. We show that for any  $r>0$
\[
\fint_{B_{r}(x)}\chi_{\Omega}|\nabla u_{\kappa}|^{2}d\mu\leq  M^{2}. 
\]
Now, if $r \leq 1$, then $B_{r}(x)\subset B_{2}$, using the fact that $\|\nabla v_{\kappa}\|_{L^{\infty}(\Omega\cap B_{r}(x))} \leq \|\nabla v_{\kappa}\|_{L^{\infty}(\Omega_{2})} \leq C_{0}$, 
\[
\fint_{B_{r}(x)}\chi_{\Omega}|\nabla u_{\kappa}|^{2}d\mu\leq  2\fint_{B_{r}(x)}\chi_{\Omega_{2}}|\nabla u_{\kappa}-\nabla v_{\kappa}|^{2} d\mu+ 2\fint_{B_{r}(x)}\chi_{\Omega_{2}}|\nabla v_{\kappa}|^{2}d\mu \leq 4 C_{0}^{2}. 
\]
If $r > 1$, then note first that $B_{r}(x) \subset B_{3r}(x_0)$ and, so scaling the first inequality in \eqref{level-set-eq1-B} by $\kappa > 1$ we obtain that 
\[\Omega
\fint_{B_{r}(x)}\chi_{\Omega}|\nabla u_{\kappa}|^{2}d\mu(x) \leq \frac{\mu(B_{3r}(x_{0}))}{\mu(B_{r}(x))}\fint_{B_{3r}(x_0)} \chi_{\Omega}|\nabla u_{\kappa}|^{2}d\mu(x) <M_{0}3^{2n}. 
\]
{\it Finally, }  set $\varpi = \max\{\kappa M\} $. Then since $\varpi >M$ we have that 
\[
\begin{split}
\mu( \{ x\in \Omega_{1}: \mathcal{M}^{\mu}(\chi_{\Omega}|\nabla u|^{2}) > \varpi^{2} \})& \leq \mu(\{x\in \Omega_{1}: \mathcal{M}^{\mu}(\chi_{\Omega} |\nabla u_{\kappa}|^{2}) > M^{2} \})\\
& \leq \mu (\{x\in \Omega_{1}: \mathcal{M}^{\mu}(\chi_{\Omega_{2}}|\nabla u_{\kappa} -\nabla v_{\kappa}|^{2})>C_{0}^{2}\})\\
& \leq \frac{C(n, M_{0})}{C_0^2}\ \mu(B_{2}) \fint_{B_{2}}\chi_{\Omega_{2}} |\nabla u_{\kappa} - \nabla v_{\kappa}|^{2} d\mu \\
&\leq C\ \eta\ \mu(B_{2}) \\
&\leq C \,M_{0} 2^{2n} \eta\ \mu(B_{1}),
\end{split}
\]
where $C(n, M_{0})$ comes from the weak $1-1$ estimates in the $\mu$ measure. 
Now we choose $\eta > 0$ small, and along the way $\delta  = \delta(\eta)$ such that $ C M_{0} 2^{2n} \eta < \epsilon$. 
\end{proof}
By scaling and translating, we can prove the following result by using Lemma \ref{level-Omega-4}. 
\begin{lemma} \label{B-scaling-level}
Suppose that $M_{0}>0$ and $\mu \in A_{2}$ such that $[\mu]_{A_{2}} \leq M_{0}$. 
There exists a constant $\varpi = \varpi(n, \Lambda, M_0)> 1$ such that the following holds true. Corresponding to any $\epsilon > 0 $, there exists a small constant $\delta= \delta(\epsilon, \Lambda, M_0,n)$ such that for any $u\in W^{1, 2}_{0}(\Omega, \mu)$ a weak solution to corresponding to  $\mathbb{A}\in \mathcal{A}_{6r}(\delta, \mu, \Lambda, \Omega)$,  any $y = (y', y_{n})\in \Omega$ and $r > 0$ with 
\[
B_{6r}^{+}(y)\subset \Omega_{6r}(y)\subset B_{6r}(y)\cap\{x_{n} > y_{n}-12r\delta\}, 
\]
and that 
\begin{equation}\label{max-f-small-forallrho-B}
\Omega_{r}(y)\cap \{x\in \mathbb{R}^{n}: \mathcal{M}^{\mu}(\chi_{\Omega}|\nabla u|^{2}) \leq 1 \}\cap \{x\in \mathbb{R}^{n}: \mathcal{M}^{\mu}\left(\left|\frac{{\bf F}}{\mu}\right|^{2}\chi_{\Omega}\right) \leq \delta^{2} \} \neq \emptyset, 
\end{equation}
then 
\[
\mu(\{x\in\mathbb{R}^{n}: \mathcal{M}^{\mu}(\chi_{\Omega}|\nabla u|^{2})  > \varpi^{2}\}\cap  \Omega_{r}(y)) < \epsilon \mu(B_{r}(y)). 
\]
\end{lemma}
The following proposition on the density of level sets is the main result of the section.
\begin{proposition}\label{contra-theorem-global}
Suppose that $M_{0}>0$ and $\mu \in A_{2}$ such that $[\mu]_{A_{2}} \leq M_{0}$. 
Let  $\varpi = \varpi(n, \Lambda, M_0)> 1$ validate Proposition \ref{theorem-contra-interior} and Lemma \ref{B-scaling-level}. For any $\epsilon >0$, there exists $\delta = \delta(\epsilon, \Lambda, M_0, n)> 0$ such that the following holds true. Suppose that $u\in W^{1,2}_{0}(\Omega)$ is a weak solution, $\Omega$ is $(\delta, R)$-Reifenberg flat and $\mathbb{A}\in \mathcal{A}_{R}(\delta, \mu,\Lambda, \Omega)$ for some $R>0$. Then if $y\in \overline{\Omega}$, $r > 0$ such that $0 < r < R/1000$ and 
\begin{equation}\label{bdry-cond}
\mu(\{x\in\mathbb{R}^{n}: \mathcal{M}^{\mu}(\chi_{\Omega}|\nabla u|^{2})  > \varpi^{2}\}\cap  \Omega_{r}(y)) \geq \epsilon \mu(B_{r}(y))
\end{equation}
then
\[
\Omega_{r}(y)\subset  \{x\in \mathbb{R}^{n}: \mathcal{M}^{\mu}(\chi_{\Omega}|\nabla u|^{2}) >1 \}\cup \{x\in \mathbb{R}^{n}: \mathcal{M}^{\mu}\left(\left|\frac{{\bf F}}{\mu}\right|^{2}\chi_{\Omega}\right) > \delta^{2} \}.
\]
\end{proposition}
%=================
\begin{proof} Note that if $B_{8r}(y)\subset \Omega$, then the result is precisely Proposition \ref{theorem-contra-interior}.  Therefore, we only need to prove this proposition when $B_{8r}(y)\cap\partial \Omega \neq \emptyset$. We argue by contradiction. Assume that the proposition is false. Then there exists a constant $\epsilon_{0} > 0$ such that corresponding to each $\delta$, we can find a $(\delta, R)$ Reifenberg flat domain $\Omega$ with some $R>0$, a coefficient $\mathbb{A}\in \mathcal{A}_{R}(\delta, \mu, \Lambda, \Omega)$, a solution $u \in W^{1,2}_0(\Omega, \mu)$, and some $r \in (0, R/1000)$, $y \in \overline{\Omega},  x_{0}\in \Omega_{r}(y)$ such that
 \eqref{bdry-cond} holds and 
\[
\mathcal{M}^{\mu}(\chi_{\Omega}|\nabla u|^{2})(x_{0}) \leq 1,\quad \text{and} \,\mathcal{M}^{\mu}\left(\left|\frac{{\bf F}}{\mu}\right|^{2}\chi_{\Omega}\right)(x_{0}) \leq \delta^{2}. 
\]
Now, let $\epsilon = \frac{\epsilon_{0}}{M_{0}144^{2n}}$, and with this $\epsilon$ we choose $\delta' < 1/7$ be as in Lemma \ref{B-scaling-level}. Let $\delta  = \frac{\delta'}{1 + \delta'}$, and corresponding to this $\delta$, let $\Omega, r >0, y, x_0$  and $u \in W^{1,2}_0(\Omega, \mu)$ be as in the above statement.  For $y_{0}\in \partial \Omega\cap B_{8r}(y)$, we observe that 
\[
x_{0} \in \Omega_r(y) = B_{r}(y)\cap \Omega \subset \Omega_{9r}(y_0) =  B_{9r}(y_{0})\cap\Omega.
\]
Also, let $M=432r$ and $\rho = \frac{M(1-\delta)}{6}$. We observe that $6\rho < R(1-\delta)$. Therefore, since $\Omega$ is $(\delta, R)$ Reifenberg flat domain, there exists a coordinate system $\{z_1, z_2,\cdots, z_n\}$ in which 
\[
y_{0} = -\delta M (1-\delta) z_{n}\in \partial \Omega, \quad y = \hat{z},\quad x_{0} = z_{0},
\]
and 
\[
B_{6\rho}^{+}(0)\subset \Omega_{6\rho} \subset B_{6\rho} \cap \{z_{n} > -12\rho\delta'\}. 
\]
We claim that $z_0 \in B_{\rho}(0)$.  Indeed, in the new coordinate system $|\hat{z}| < 8r + \delta M$, and therefore 
\[
|z_{0}| \leq 9r + \delta M \leq \rho.
\]
In summary, up to a change of coordinate system, and after a simple calculation 
\begin{itemize}
\item[(i)] $u\in W^{1, 2}_{0}(\Omega, \mu)$ is a weak solution to \eqref{main-eqn}, 
\item[(ii)] $B_{6\rho}^{+}(0)\subset \Omega_{6\rho}\subset B_{6\rho} \cap \{z_{n}> -12\rho \delta'\}$, 
\item[(iii)] $z_{0}\in B_{\rho}(0)\cap \{ \mathcal{M}^{\mu}(\chi_{\Omega} |\nabla u|^{2}) \leq 1 \}\cap \{ \mathcal{M}(\chi_{\Omega} \left|\frac{{\bf F}}{\mu}\right|^{2}) \leq \delta'^{2}\}, and $
\item[(iv)] $B_{r}(y) \subset B_{\rho}(0)\subset B_{144r}(y)$.
\end{itemize}
Thus, from items $(i)$-$(iii)$ we see that all the hypotheses of Lemma \ref{B-scaling-level} are satisfied with $B_{\rho}(0)$ replacing $B_{r}(y)$. We thus conclude that 
\[
\mu(\Omega_{\rho}(0)\cap \{\mathcal{M}^{\mu}(\chi_{\Omega} |\nabla u|^{2}) > \varpi^{2} \} ) <  \epsilon  \mu(B_{\rho}(0)). 
\]
Moreover, from item $(iv)$ we have that 
\[
\begin{split}
\mu(\Omega_{r}(y)\cap \{\mathcal{M}^{\mu}(\chi_{\Omega} |\nabla u|^{2}) > \varpi^{2} \} ) &\leq\mu(\Omega_{\rho}(0)\cap \{\mathcal{M}^{\mu}(\chi_{\Omega} |\nabla u|^{2}) > \varpi^{2} \} )\\
&<  \frac{\epsilon_{0}}{M_{0}144^{2n}}\mu(B_{\rho}(0)) \leq \frac{\epsilon_{0}}{M_{0}144^{2n}}\mu(B_{144r}(y)) \\
&\leq  \frac{\epsilon_{0}}{M_{0}144^{2n}}M_{0}144^{2n} \mu(B_{r}(y)) \\
&= \epsilon_{0}\mu(B_{r}(y)),
\end{split}
\]
where we have used the doubling property of the $\mu$. The last sequence of inequalities obviously contradict the hypothesis \eqref{bdry-cond} of the theorem, and thus the proof is complete. 
\end{proof}
\subsection{Proof of the global $W^{1,p}$-regularity estimates} \label{proof-main-th}
Our  first statement of the subsection, which is the key in obtaining the higher gradient integrability of solution, gives the level set estimate of $\mathcal{M}^{\mu}(\chi_{\Omega}|\nabla u|^{2})$ in terms that of $\mathcal{M}^{\mu}\left(\left|\frac{{\bf F}}{\mu}\right|^{2} \chi_{\Omega}\right)$. 
\begin{lemma}\label{mainCor-level} Let $M_{0}>0$ and $\mu \in A_{2}$ such that $[\mu]_{A_{2}} \leq M_{0}$. 
Let $\varpi = \varpi(n, \Lambda, M_0)> 1$ be the constant defined in Proposition \ref{contra-theorem-global}. 
For a given  $\epsilon > 0 $, there is $\delta= \delta(\epsilon, \Lambda, M_0) < 1/4$  such that  for any   $u\in W^{1, 2}_{0}(\Omega, \mu)$ a weak solution of \eqref{main-eqn} to with 
$\mathbb{A}\in \mathcal{A}_{R}(\delta, \mu, \Lambda, \Omega)$, and $\Omega$ a $(\delta, R)$ Reifenberg flat domain, and for a fixed $0< r_{0} < R/2000$, if 
\[
\mu( \{x\in \Omega: \mathcal{M}^{\mu}(\chi_{\Omega}|\nabla u|^{2})  > \varpi^{2}\}) <\epsilon \mu(B_{r_{0}}(y)), \quad
\forall y \in \overline{\Omega}, 
\]
then for any $k\in \mathbb{N}$ and $\epsilon_{1} = \left(\frac{10}{1 - 4\delta}\right)^{2n} M_{0}^{2} \epsilon$, we have that 
\[
\begin{split}
 \mu(\{x\in \Omega: \mathcal{M}^{\mu}(\chi_{\Omega}|\nabla u|^{2}) > \varpi^{2k} \}) &\leq \sum_{i=1}^{k} \epsilon_{1}^{i} \mu\left(\{x\in \Omega: \mathcal{M}^{\mu}\left(\left|\frac{{\bf F}}{\mu}\right|^{2} \chi_{\Omega}\right) >\delta^{2} \varpi^{2(k-i)} \}\right)\\
&\quad\quad+ \epsilon_{1}^{k}\mu(\{x\in \Omega: \mathcal{M}^{\mu}(\chi_{\Omega} |\nabla u|^{2}) > 1 \}). 
\end{split}
\]
\end{lemma}
 \begin{proof}
 We will use induction on $k$. For the case $k=1$, we are going to apply Lemma \ref{Vitali}, by taking 
 \[C = \{x\in \Omega: \mathcal{M}^{\mu}(\chi_{\Omega}|\nabla u|^{2}) > \varpi^{2} \}, 
 \] 
and 
\[
D = \{x\in  \Omega: \mathcal{M}^{\mu}\left(\left|\frac{{\bf F}}{\mu}\right|^{2} \chi_{\Omega}\right) >\delta^{2}  \}\cup \{x\in  \Omega: \mathcal{M}^{\mu}(\chi_{\Omega} |\nabla u|^{2}) > 1 \}.
\]
By assumption,  $\mu(C)< \epsilon  \mu(B_{r_{0}} (y))$ for all $y\in \overline{\Omega}$. Also for any $y\in \Omega$ and $\rho\in (0,  2r_{0}) $, then $\rho\in (0, R/1000)$ and if 
$\mu (C \cap B_{\rho}(y))\geq \epsilon \mu((B_{\rho}(y))$, then   by Proposition \ref{contra-theorem-global} we have that 
\[
B_{\rho}(y)\cap \Omega \subset D.
\]
Thus, all the conditions of Lemma \ref{Vitali} are satisfied and therefore  we have 
\[
\mu(C) \leq \epsilon_{1} \mu(D).
\]
which is exactly the case when $k=1$. Assume it is true for $k$. We will show the statement for $k+1$. We normalize $u$ to $u_{\varpi} = u/\varpi$ and ${\bf F}_{\varpi} = {\bf F}/\varpi$, and we see that since $\varpi > 1$ we have 
\[
\begin{split}
& \mu(\{x\in \Omega: \mathcal{M}^{\mu}(\chi_{\Omega}|\nabla u_{\varpi}|^{2}) > \varpi^{2} \}) \\
&= \mu(\{x\in \Omega: \mathcal{M}^{\mu}(\chi_{\Omega}|\nabla u|^{2}) > \varpi^{4} \})\\
& \leq \mu(\{x\in \Omega: \mathcal{M}^{\mu}(\chi_{\Omega}|\nabla u|^{2}) > \varpi^{2} \}) \leq \epsilon \mu(B_{r_{0}}(y)), \quad \forall  y \in \overline{\Omega}.
\end{split}
\]
By induction assumption, it follows then that 
\[
\begin{split}
 \mu(\{x\in \Omega: \mathcal{M}^{\mu}(\chi_{\Omega}|\nabla u|^{2}) > \varpi^{2(k+1)} \})&=
 \mu(\{x\in \Omega: \mathcal{M}^{\mu}(\chi_{\Omega}|\nabla u_{\varpi}|^{2}) > \varpi^{2k} \})\\
  &\leq \sum_{i=1}^{k} \epsilon_{1}^{i} \mu\left(\{x\in \Omega: \mathcal{M}^{\mu}\left(\left|\frac{{\bf F}_{\varpi}}{\mu}\right|^{2} \chi_{\Omega}\right) >\delta^{2} \varpi^{2(k-i)} \}\right)\\
&\quad\quad+ \epsilon_{1}^{k}\mu(\{x\in \Omega: \mathcal{M}^{\mu}(\chi_{\Omega} |\nabla u_{\varpi}|^{2}) > 1 \})\\
& =  \sum_{i=1}^{k} \epsilon_{1}^{i} \mu\left(\{x\in \Omega: \mathcal{M}^{\mu}\left(\left|\frac{{\bf F}}{\mu}\right|^{2} \chi_{\Omega}\right) >\delta^{2} \varpi^{2(k+1-i)} \}\right)\\
&\quad\quad+ \epsilon_{1}^{k}\mu(\{x\in \Omega: \mathcal{M}^{\mu}(\chi_{\Omega} |\nabla u|^{2}) > \varpi^2 \}). 
\end{split}
\]
Applying the case $k=1$ to the last term we obtain that 
\[
\begin{split}
 \mu(\{x\in \Omega: \mathcal{M}^{\mu}(\chi_{\Omega}|\nabla u|^{2}) > \varpi^{2(k+1)} \})&\leq \sum_{i=1}^{k + 1} \epsilon_{1}^{i} \mu\left(\{x\in \Omega: \mathcal{M}^{\mu}\left(\left|\frac{{\bf F}}{\mu}\right|^{2} \chi_{\Omega}\right) >\delta^{2} \varpi^{2(k+1-i)} \}\right)\\
&\quad\quad+ \epsilon_{1}^{k+1}\mu(\{x\in \Omega: \mathcal{M}^{\mu}(\chi_{\Omega} |\nabla u|^{2}) > 1 \}), 
\end{split}
\]
as desired. 
 \end{proof} 
 \begin{proof}[Completion of the proof of Theorem \ref{g-theorem}] \label{proof-main-result}
We divide the proof in two parts based on whether $p\geq 2$ or $1 < p < 2$. \\ \ \\
{\bf Case 1}: $p\geq 2$. For this case, $A_2 \subset A_p$ and therefore $ \mu \in A_2\cap A_p = A_2$. Moreover, because ${\bf F}/\mu\in L^{p}(\Omega,\mu) $, then clearly ${\bf F}/\mu\in L^{2}(\Omega,\mu)$. Applying 
\cite[Theorem 2.2] {Fabes}, a unique solution $u\in W^{1, 2}_{0}(\Omega,\mu)$ of \eqref{main-eqn} exists. Moreover, it follows by the energy estimate that  
\[
\|\nabla u\|_{L^{2}(\Omega,\mu)} \leq C(\Lambda) \left\|\frac{{\bf F}}{\mu}\right\|_{L^{2}(\Omega,\mu)}. 
\] 
Our goal is to show that $\nabla u\in L^{p}(\Omega,\mu)$.  Let $\epsilon>0$ be given, then $\delta>0$ is chosen according to Lemma \ref{mainCor-level}. Also, take $r_{0} = R/2000$ and a ball $B = B_{s}(0)$ with sufficiently large $s$ depending only on $\text{diam}(\Omega), r_0$ so that 
\[  B_{r_0}(y) \subset B, \quad \forall \ y \in \overline{\Omega}.
\]
Then by doubling property of $\mu$ \eqref{doubling} we have 
\[
\mu(\Omega)\leq \mu(B) \leq  M_{0} \left(\frac{|B|}{|B_{r_{0}}(y)|}\right)^{2} \mu(B_{r_{0}}(y)) = 
M_0 \left( \frac{s}{r_0}\right)^{2n}  \mu(B_{r_{0}}(y)) ,\quad \forall y\in \overline{\Omega}. 
\]
We claim we can choose $N$ large such that for $u_{N} = u/N$, 
\[
\mu(\{x\in \Omega: \mathcal{M}^{\mu}(\chi_{\Omega} |\nabla u_{N}|^{2}) > \varpi^{2} \})\leq \epsilon \mu( B_{r_{0}} (y)),\quad \forall y\in \overline{\Omega}. 
\]
To see this we first assume that $\|\nabla u\|_{L^{2}(\Omega,\mu)} >0$.  Then by weak (1,1) estimate for maximal functions there exists a constant $C(n, M_{0})>0$ such that 
\[
\mu(\{x\in \Omega: \mathcal{M}^{\mu}(\chi_{\Omega} |\nabla u_{N}|^{2}) > \varpi^{2} \}) \leq \frac{C (n, M_{0})}{N^{2} \varpi^{2}} \int_{\Omega}|\nabla u|^{2}d\mu.
\]
Then, the claim follows if we select $N$ such that 
\begin{equation*}
\frac{C(n, M_{0})}{N^{2} \varpi^{2}} \int_{\Omega}|\nabla u|^{2}d\mu  =\epsilon \frac{\mu(B)}{M_{0} \left(\frac{s}{r_0} \right)^{2n}}. 
\end{equation*}
We observe that by the doubling property of $\mu$, it follows from the above estimate that
\begin{equation} \label{muBzero-B}
N^2 \mu(\Omega) \leq C(n, M_0, \text{diam}(\Omega))\int_{\Omega} |\nabla u|^2 \mu(x) dx.
\end{equation}
Now consider the sum 
\[
S = \sum_{k=1}^{\infty} \varpi^{pk}\mu(\{\Omega: \mathcal{M}(\chi_{\Omega}|\nabla u_{N}|^{2} ) (x) > \varpi^{2k}\}). 
\]
Applying the previous corollary we have that 
\[
\begin{split}
S &\leq \sum_{k=1}^{\infty} \varpi^{pk} \left[\sum_{i=1}^{k } \epsilon_{1}^{i} \mu\left(\{x\in \Omega: \mathcal{M}^{\mu}\left(\left|\frac{{\bf F}_{N}}{\mu}\right|^{2} \chi_{\Omega}\right) >\delta^{2} \varpi^{2(k-i)} \}
\right)\right]
\\
&\quad  \quad +\sum_{k=1}^{\infty} \varpi^{pk} \epsilon_{1}^{k}\mu(\{\Omega: \mathcal{M}^{\mu}(\chi_{\Omega}|\nabla u_{N}|^{2} ) (x) > 1\}).
\end{split}
\]
Applying summation by parts we have that 
\[
\begin{split}
S &\leq \sum_{j=1}^{\infty}( \varpi^{p} \epsilon_{1})^{j} \left[\sum_{k=j}^{\infty } \varpi^{p(k-j)} \mu\left(\{x\in \Omega: \mathcal{M}^{\mu}\left(\left|\frac{{\bf F}_{N}}{\mu}\right|^{2} \chi_{\Omega}\right) >\delta^{2} \varpi^{2(k-j)} \}
\right)\right]
\\
&\quad  \quad +\sum_{k=1}^{\infty} (\varpi^{p} \epsilon_{1})^{k}\mu(\{\Omega: \mathcal{M}^{\mu}(\chi_{\Omega}|\nabla u_{N}|^{2} ) (x) > 1\})\\
&\leq C\left(\left\| \mathcal{M}^{\mu}\left(\chi_{\Omega}\left|\frac{{\bf F}_{N}}{\mu}\right|^{2}\right)\right\|^{p/2}_{L^{p/2}(\Omega, \mu)}  + \|\nabla u_{N}\|^{2}_{L^{2}(\Omega, \mu)}\right) \sum_{k=1}^{\infty} (\varpi^{p} \epsilon_{1})^{k}\\
\end{split}
\]
where we have applied the weak $(1, 1)$ estimate of the maximal function $\mathcal{M}^{\mu}$. Now chose $\epsilon $ small so that $\varpi^{p} \epsilon_{1}  <1$ to obtain that 
\[
S \leq C\left(\left\| \mathcal{M}^{\mu}\left(\chi_{\Omega}\left|\frac{{\bf F}_{N}}{\mu}\right|^{2}\right)\right\|^{p/2}_{L^{p/2}(\Omega, \mu)}  + \|\nabla u_{N}\|^{2}_{L^{2}(\Omega, \mu)}\right) \leq C\left(\left\|\frac{{\bf F}_{N}}{\mu}\right\|^{p}_{L^{p}(\Omega, \mu)}  + \|\nabla u_{N}\|^{2}_{L^{2}(\Omega, \mu)}\right),
\]
where we have applied the strong $(p,p)$ estimate for the maximal function operator $\mathcal{M}^{\mu}$. 

Applying Lemma \ref{measuretheory-lp}, we have  
\[
\|\nabla u_{N} \|_{L^{p}(\Omega, \mu)}^{p} \leq C \|\mathcal{M}(\chi_{\Omega}|\nabla u_{N}|^{2})\|^{p/2}_{L^{p/2}(\Omega, \mu)} \leq   C (S + \mu(\Omega)), 
\]
and therefore multiplying by $N^{p}$ and using formula \eqref{muBzero-B} we have 
\begin{equation} \label{Lp-last}
\begin{split}
\|\nabla u \|_{L^{p}(\Omega, \mu)}^{p} &\leq C\left(\left\|\frac{{\bf F}}{\mu}\right\|^{p}_{L^{p}(\Omega, \mu)}  +N^{p-2} \|\nabla u\|^{2}_{L^{2}(\Omega, \mu)}+ N^{p}\mu(\Omega) \right)\\
& \leq C\left(\left\|\frac{{\bf F}}{\mu}\right\|^{p}_{L^{p}(\Omega, \mu)}  +N^{p}\mu(\Omega) \right).
\end{split}
\end{equation}
Finally we estimate $N^{p}\mu(\Omega)$ using formula \eqref{muBzero-B} and H\"older's inequality  together with the energy estimate as 
\[
N^{p}\mu(\Omega) \leq C(n, M_0, \text{diam}(\Omega)) N^{p-2} \|\nabla u\|_{L^{2}(\Omega,\mu)}^{2} \leq C N^{p-2}\left\|\frac{{\bf F}}{\mu}\right\|^{2}_{L^{2}(\Omega, \mu)} \leq C \left\|\frac{{\bf F}}{\mu}\right\|^{2}_{L^{p}(\Omega, \mu)} [N^{p}\mu(\Omega)]^{\frac{p-2}{p}}.
\]
This estimate implies
\[
N^{p}\mu(\Omega)  \leq C (n, M_0, \text{diam}(\Omega)) \left\|\frac{{\bf F}}{\mu}\right\|^{p}_{L^{p}(\Omega, \mu)}.
\]
By pluging the last estimate into \eqref{Lp-last}, we obtain the desired estimate \eqref{main-est}.  \\ \ \\
%================
{\bf Case 2}: $1 < p < 2$. In this case, $\mu \in A_2 \cap A_p = A_p$. We use the standard duality argument. Suppose that ${\bf F}/\mu\in L^{p}(\Omega,\mu)$. By density of $C_{c}^{\infty}(\Omega)$ in $L^{p}(\Omega,\mu)$, there exists  a sequence of functions ${\bf f}_{n}\in C_{c}^{\infty}(\Omega) $ such that ${\bf f}_{n} \to {\bf F}/\mu$ in $L^{p}(\Omega,\mu)$. Corresponding to each ${\bf f}_{n}$, there exists $u_{n}\in W^{1, 2}_{0}(\Omega,\mu)$ solving the equation $\text{div} (\mathbb{A}(x)\nabla u_{n}) = \text{div}(\mu {\bf f}_{n})$ with the estimate 
\[
\|u_{n}\|_{W^{1, 2}_{0}(\Omega,\mu)}\leq C \|{\bf f}_{n}\|_{L^{2}(\Omega,\mu)}
\] 
where $C>0$ is independent of $u_{n}$ and ${\bf f}_{n}$.  Since $p\in (1, 2)$,  $u_{n}\in W^{1, p}_{0}(\Omega,\mu)$ for all n. We claim that $u_{n}$ is in fact bounded in $W^{1, p}_{0}(\Omega,\mu)$. To that end, let ${\bf g}\in L^{p'}(\Omega,\mu)$ be given with $\|{\bf g}\|_{L^{p'}(\Omega,\mu)} \leq 1$, where $p'$ is the H\"older conjugate of $p$. Since $p' > 2$,  by case 1, there exists a small constant $\delta > 0$ such that whenever $\mathbb{A}\in \mathcal{A}_{R_{0}}(\delta, \mu, \Lambda, \Omega)$, there exists a function $w\in W^{1, p'}_{0}(\Omega,\mu)$ that solves 
$$\text{div} (\mathbb{A}(x)\nabla w) = \text{div}(\mu {\bf g})\quad \text{in $\Omega$} $$ 
weakly accompanied by the estimate 
\[
\|\nabla w\|_{W^{1, p'}_{0}(\Omega,\mu)}\leq C \|{\bf g}\|_{L^{p'}(\Omega,\mu)}, 
\] 
where $C$ is independent of $w$ and ${\bf g}$. Now, we have that 
\[
\begin{split}
\int_{\Omega}\langle \nabla u_{n}(x), {\bf g}(x) \rangle d\mu(x)  &= \int_{\Omega}\langle \nabla u_{n}(x),\mu {\bf g}(x) \rangle dx\\
& = \int_{\Omega}\langle \nabla u_{n}(x), \mathbb{A}(x)\nabla w \rangle dx = \int_{\Omega}\langle \mathbb{A}(x)\nabla u_{n}(x), \nabla w \rangle dx = \int_{\Omega}\langle {\bf f}_{n}, \nabla w \rangle d\mu(x),
\end{split}
\]
where we have used the fact for each $n$, the function $u_{n}$, is an allowable test function for the equation involving $w$ and vice versa. 
And therefore by the definition of the dual norm, 
\[
\begin{split}
\|\nabla u_{n}\|_{L^{p}(\Omega,\mu)} = \sup_{\|{\bf g}\|_{L^{p'}(\Omega,\mu)} \leq 1 }\left|\int_{\Omega}\langle\nabla u_{n}(x), {\bf g}(x)\rangle d\mu(x)\right| &\leq \sup_{\|{\bf g}\|_{L^{p'}(\Omega,\mu)} \leq 1 }\|{\bf f}_{n}\|_{L^{p}(\Omega,\mu)} \|\nabla w\|_{L^{p'}(\Omega,\mu)} \\
&\leq C \|{\bf f}_{n}\|_{L^{p}(\Omega,\mu)} \leq C \left\|{\bf F}/\mu\right\|_{L^{p}(\Omega,\mu)}. 
\end{split}
\]
Therefore by Poincare's inequality, which we can apply because $\mu \in A_{p}$,  $u_{n}$ is bounded in $W^{1, p}_{0}(\Omega,\mu)$, and thus has a weak limit $u$ in $W^{1, p}_{0}(\Omega,\mu)$. Clearly $u$ solves the equation $
\text{div} (\mathbb{A}(x)\nabla u) = \text{div}({\bf F})
$ weakly. Moreover, we also have   the estimate 
\[
\|\nabla u\|_{L^{p}(\Omega,\mu)}  \leq \liminf_{n\to \infty} \|\nabla u_{n}\|_{L^{p}(\Omega,\mu)} \leq C  \left\|{\bf F}/\mu\right\|_{L^{p}(\Omega,\mu)}
\]
as desired. 
What is left is to show the above solution $u$ is unique. To that end it suffices to show that if $u\in W^{1,p}_{0}(\Omega,\mu)$ and $
\text{div} (\mathbb{A}(x)\nabla u) =0, 
$ then $u=0.$ To show this, we begin by noting that $|\nabla u|^{p-2}\nabla u \in L^{p'}(\Omega,\mu)$, $p'>2$ and that there is a weak solution $w\in W^{1, p'}_{0}(\Omega,\mu)$ to 
$\text{div} (\mathbb{A}(x)\nabla w) = \text{div}(\mu |\nabla u|^{p-2}\nabla u)$. Using $u$ as a test function for the equation of $w$ and vice versa, we obtain that
\[
\int_{\Omega} |\nabla u|^{p}d\mu = \int_{\Omega}\langle \mu |\nabla u|^{p-2}\nabla u, \nabla u\rangle dx = \int_{\Omega} \langle \mathbb{A}(x)\nabla w, \nabla u \rangle dx = \int_{\Omega} \langle\mathbb{A}(x)\nabla u, \nabla w \rangle dx = 0. 
\]
This implies that $\nabla u = 0$ a.e.  and therefore $u\equiv 0$. That concludes the proof. 
\end{proof}
\appendix 
\section{Proof of Lemma \ref{Example2-combinedLemma} }
The proof of Lemma \ref{Example2-combinedLemma} follows from the following two lemmas. 
The first lemma implies that $[\mu]_{A_2}$ can be bounded by a uniform constant depending only on $n$ if $|\alpha| \leq 1$.
\begin{lemma}  \label{A-2-alpha}
Let  $\mu (x) = |x|^{\alpha}$ for  $x \in \mathbb{R}^n$ and $|\alpha| < n$.  Then  we have the estimate  
\[
[\mu]_{A_{2}} \leq \max\left\{ 2^{n}5^{|\alpha|}, \frac{2^{4n}}{(n+ \alpha)(n-\alpha)}\right\}.\]
In particular, if $|\alpha| \leq n_{0} < n$, then we can bound $[\mu]_{A_{2}}$ from above independent of $\alpha$ as  $[\mu]_{A_{2}} \leq  \max\left\{  2^{n+3} , \frac{2^{4n}}{n^2 - n_{0}^2}\right\}. 
$
\end{lemma}
\begin{proof}
Following the calculations in \cite{Grafakos} we classify balls $B_{r}(x_{0})$ as type I if $|x_{0}| \geq 3r$ and type II if $|x_{0}|\leq 3r$.  
For type I balls, we have that $||x_{0}| + 2r| \leq 4 ||x_{0}| -r|$ and $||x_{0}| -2r| \geq \frac{1}{4}||x_{0}| + r|$. 
Then it follows from \cite[Example 9.1.6]{Grafakos} that 
\[
\begin{split}
\fint_{B_{r}(x_{0})}\mu(x)dx \fint_{B_{r}(x_{0})} \mu(x)^{-1} dx  &= \frac{1}{|B_{r}(x_{0})|^{2}} \int_{B_{r}(x_{0})}|x|^{\alpha}dx \int_{B_{r}(x_{0})} |x|^{-\alpha} dx \\
&\leq \frac{1}{|B_{r}(x_{0})|^{2}} \int_{B_{2r}(x_{0})}|x|^{\alpha}dx \int_{B_{2r}(x_{0})} |x|^{-\alpha} dx\\
& \leq 2^{n} \left\{\begin{split} \left(\frac{x_{0} + 2r}{x_{0} - 2r}\right)^{\alpha}, & \quad\text{if $\alpha \in[0, n)$}\\
\left(\frac{x_{0} - 2r}{x_{0} + 2r}\right)^{\alpha},&\quad \text{if $\alpha \in (-n, 0)$}
\end{split}\right.\\
&= 2^{n} \left( 1  + \frac{ 4r}{x_{0} - 2r}\right)^{|\alpha|} \leq 2^{n}5^{|\alpha|}. 
\end{split}
\]
For type II balls, $B_{r}(x_{0}) \subset B_{4r}(0)$ and therefore we have 
\[
\begin{split}
\fint_{B_{r}(x_{0})}\mu(x)dx \fint_{B_{r}(x_{0})} \mu(x)^{-1} dx  &= \frac{1}{|B_{r}(x_{0})|^{2}} \int_{B_{4r}(0)}|x|^{\alpha}dx \int_{B_{4r}(0)} |x|^{-\alpha} dx \\
&\leq \frac{\nu_{n}^{2}}{|B_{r}(x_{0})|^{2}}\frac{(4r)^{n+  \alpha}}{n + \alpha} \frac{(4r)^{n-\alpha}}{n-\alpha} \\
&=  \frac{\nu_{n}^{2}}{|B_{r}(x_{0})|^{2}}\frac{(4r)^{2n}}{(n + \alpha)(n-\alpha)} = \frac{2^{4n}}{(n + \alpha)(n-\alpha)}. 
\end{split}
\]
\end{proof}
%====================

\begin{lemma} \label{Example-1-lemma} Let $\mu(x) = |x|^\alpha$ for $x \in \mathbb{R}^n$ and $|\alpha| < 1$.  Then, 
\[
\int_{B_{r} (x_0)} \Big| \mu(x) - \bar{\mu}_{B_{r}(x_0)}\Big | dx \leq \frac{2|\alpha|4^{2n+1} }{2n-1}  \int_{B_r(x_0)} \mu(x) dx, \quad \forall x_0 \in \mathbb{R}^n, \quad \forall r >0.
\]
\end{lemma}
\begin{proof} We first need to perform some elementary calculations. Note that
\begin{equation} \label{mu-ball-0}
\mu(B_r(0)) = \int_{B_r(0)} |x|^\alpha dx = \omega_n \int_{0}^r s^{n+\alpha -1} ds = 
\frac{\omega_n r^{n+\alpha}}{n + \alpha },
\end{equation}
where $\omega_n$ is the Lebesgue measure of the unit sphere in $\mathbb{R}^n$. On the other hand, for every $r >0$, we have
\[
\begin{split}
\frac{1}{|B_r(0)|}\int_{B_r(0)}\int_{B_r (0)} \Big| |x|^\alpha  - |y|^\alpha \Big | dx  & = \frac{n}{\omega_n r^n} \int_{B_r(0)} \int_{B_r(0)} \Big | |x|^\alpha - |y|^\alpha \Big | dx dy \\
& = \frac{n \omega_n}{ r^n} \int_0^r \int_0^r |s^\alpha - t^\alpha| s^{n-1} t^{n-1} ds dt \\
& =  \frac{\sgn(\alpha)2n \omega_n}{ r^n} \int_0^r  s^{n-1}  \int_0^s (s^\alpha - t^\alpha)  t^{n-1}  dt ds  \\
& = \frac{\sgn(\alpha)2n \omega_n}{ r^n} \int_0^r  s^{n-1} \left( \frac{s^{n+\alpha}}{n} -  \frac{s^{n+\alpha}}{n+\alpha} \right) ds  \\
& = \frac{\sgn(\alpha)\alpha 2\omega_n}{ (n+\alpha) r^n} \int_0^r  s^{2n + \alpha -1}ds .
\end{split}
\]
Noting that $|\alpha| = \sgn(\alpha) \alpha$, we conclude that 
\begin{equation} \label{ocs-at-zero}
\frac{1}{|B_r(0)|}\int_{B_r(0)}\int_{B_r (0)} \Big| |x|^\alpha  - |y|^\alpha \Big | dx   =
\frac{2 |\alpha| \omega_n r^{n+ \alpha}}{ (n+\alpha) (2n +\alpha) }, \quad \forall \ r >0.
\end{equation}
The proof now is divided in two cases depending on the locations and sizes of the balls. \\
{\bf Case I:} We consider balls $B_r(x_0)$ with $ r > |x_0|/3$. In this case, note that $B_r(x_0) \subset B_{4r}(0)$. Therefore, 
\begin{equation} \label{case-I-x_0}
\begin{split}
\int_{B_r(x_0)} \Big| \mu(x) - \overline{\mu}_{B_r(x_0)} \Big| dx  & \leq \frac{1}{|B_r(x_0)|} \int_{B_r(x_0)} \int_{B_r(x_0)} \Big||x|^\alpha - |y|^\alpha \Big| dx dy \\
& \leq \frac{4^n}{|B_{4r}(x_0)|} \int_{B_{4r}(0)}  \int_{B_{4r}(0)} \Big| |x|^\alpha - |y|^\alpha \Big| dx dy \\
& = \frac{2 |\alpha| \omega_n  4^{2n + \alpha} r^{n+ \alpha}}{(n+\alpha) (2n +\alpha) },
\end{split}
\end{equation}
where in the last equality, we used \eqref{ocs-at-zero}. On the other hand, we claim that
\begin{equation}\label{weight-of-ball-at-x0}
\mu(B_r(x_0))  = \int_{B_r(x_0)} |x|^\alpha dx  \geq \left\{
\begin{split} \frac{\omega_{n} }{n + \alpha}r^{n + \alpha}, &\quad \text{if $0 \leq \alpha < 1$}, \\
\omega_{n}4^{\alpha} r^{n + \alpha}, &\quad \text{if $-1 < \alpha \leq 0$}.
\end{split}
\right.
\end{equation}
Then, by combining \eqref{mu-ball-0}, \eqref{case-I-x_0}, and \eqref{weight-of-ball-at-x0} we infer that, since $|\alpha| < 1$
\[
\int_{B_r(x_0)} \Big| \mu(x) - \overline{\mu}_{B_r(x_0)} \Big| dx \leq  \frac{|\alpha|2 \cdot  4^{2n + 1} }{2n +\alpha} \mu(B_r(x_0)).
\]
Therefore, for this case, it remains  to show \eqref{weight-of-ball-at-x0}. The case when $-1 < \alpha \leq 0$ is easy since $x\in B_{r}(x_{0})$ and $|x_{0}| < 3r$ implies that $|x| < 4r$. 
To prove the inequality for the case $0 \leq \alpha < n$, we proceed as follow. Because $\mu$ is radial, by rotation, we can assume that $x_0 = (0', a)$, where $a = |x_0| \geq 0$ and $0' $ is the origin of $\mathbb{R}^{n-1}$. Then, we write
\[
f(a): = \int_{B_r(x_0)} \mu(x) dx = \int_{B_r(0')} \int_{a - \sqrt{r^2 -|x'|^2}}^{a + \sqrt{r^2 -|x'|^2}} [|x'|^2 + (x_n)^2]^{\alpha/2} dx_n dx', 
\]
where $B_r(0')$ is the ball in $\mathbb{R}^{n-1}$ centered at $0'$. Note that since $\alpha \geq 0$, the fundamental theorem of calculus gives
\[
f'(a) =  \int_{B_r(0')} \left\{\Big [|x'|^2 + (a + \sqrt{r^2 -|x'|^2})^2\Big ]^{\alpha/2} -\Big [|x'|^2 + (a - \sqrt{r^2 -|x'|^2})^2\Big]^{\alpha/2} \right\} dx'\geq 0.
\]
Hence, $f(a) \geq f(0)$ which is  \eqref{weight-of-ball-at-x0}. \\
{\bf Case II:} We consider balls $B_r(x_0)$ with $0<  r \leq |x_0|/3$. In this case, note that since $ 0 \not\in B_r(x_0)$, $\mu$ is smooth in $B_r(x_0)$.  Let us first consider the case that $0\leq \alpha < 1$.  Applying mean value theorem for every $x, y \in B_r(x_0)$, there is $x^*$ in between $x$ and $y$ such that  
\[
\Big | |x|^\alpha - |y|^\alpha \Big | = |\alpha| |x-y| |x^*|^{\alpha -1} \leq 2 |\alpha| r |x^*|^{\alpha -1}. 
\]
Also, since $x^* \in B_r(x_0)$, we have
\[
|x^*| \geq |x_0| - |x_0 - x^*| \geq 3r -r = 2r.
\]
This together with the observation that $0\leq \alpha <1$, we obtain
\[ \Big | |x|^\alpha - |y|^\alpha \Big | \leq 2^{\alpha} |\alpha| r^{\alpha} \leq 2  |\alpha| r^{\alpha}. 
\]
Hence,
\[
\int_{B_r(x_0)} \Big |\mu (x) - \overline{\mu}_{B_r(x_0)} \Big| dx  \leq \frac{2  \omega_n |\alpha| r^{n+\alpha}}{n}. 
\]
On the other hand,  since $\alpha \geq 0$, we have 
\[
\begin{split}
\mu(B_r(x_0)) & = \int_{B_r(x_0)} |x|^\alpha dx = \int_{B_r(0)} |y + x_0|^\alpha dy 
 \geq  \int_{B_r(0)} (|x_0| - |y|)^\alpha dy \geq \int_{B_r(0)} (2r)^\alpha dy 
 =\frac{ \omega_n 2^\alpha r^{n+\alpha}}{n}.
\end{split}
\]
Combining we obtain,
\[
\int_{B_r(x_0)} \Big |\mu (x) - \overline{\mu}_{B_r(x_0)} \Big| dx   \leq 2^{1-\alpha} \alpha \mu(B_r(x_0)) \leq 
2 \alpha \mu(B_r(x_0)). 
\]
Let us do now the case $-1 < \alpha \leq 0$. As before,  for every $x, y \in B_r(x_0)$,
and $|\alpha| <1$, we have
\[ \Big | |x|^{|\alpha|} - |y|^{|\alpha|} \Big | \leq 2^{|\alpha|} |\alpha| r^{|\alpha|}  \]

Notice also that $|x| \ge |x_0| -|x-x_0| \ge 3r -r =2r$ and $|y| \ge |x_0| -|y-x_0| \ge |x_0| -r 
>0$ , thus 
\[ \Big | {|x|^\alpha} - {|y|^\alpha} \Big |= \Big | \frac{|x|^{|\alpha|}-|y|^{|\alpha|}}{|x|^{|\alpha|} |y|^{|\alpha|}} \Big | \le \frac{ 2^{|\alpha|}  |\alpha| r^{\alpha}}{(2r)^{|\alpha|} (|x_0| -r )^{|\alpha|}}= \frac{ |\alpha| }{(|x_0| -r )^{|\alpha| }}.\]
Therefore, noting that $|\alpha| = -\alpha$
\[
\begin{split}
\int_{B_{r}(x_0)} \Big| \mu(x) - \overline{\mu}_{B_{r}(x_0)}\Big | dx 
&\leq \frac{1}{|B_{r}(x_0)|}\int_{B_{r}(x_0)}\int_{B_{r}(x_0)} \Big| {|x|^{\alpha}} - {|y|^{\alpha}}\Big | dydx \\
&\leq \frac{1}{\omega_n r^n}\int_{B_{r}(x_0)}\int_{B_{r}(x_0)} \ |\alpha| {(|x_0| -r )^\alpha} dydx =  |\alpha| \omega_n r^n{(|x_0| -r )^\alpha}.
\end{split}
\] 
On the other hand,  
$$
\mu(B_{r}(x_0))  = \int_{B_{r}(x_0)}  |x|^\alpha dx \ge  \int_{B_{r}(x_0)}  {(|x_0|+r)^\alpha}dx = \omega_n r^n{(|x_0|+r)^\alpha}.
$$
It follows that 
\[
\begin{split}
& \int_{B_{r}(x_0)} \Big| \mu(x) - \overline{\mu}_{B_{r}(x_0)}\Big | dx  \leq  \frac{|\alpha| (|x_0| - r )^\alpha}{(|x_0| +r )^\alpha} \mu(B_{r}(x_0)) \\
& = |\alpha| \left(\frac{|x_0| + r }{|x_0| -r }\right)^{|\alpha|} \mu(B_{r}(x_0)) \leq 2|\alpha| \mu(B_{r}(x_0)) .
\end{split}
\]
Lemma \ref{Example-1-lemma} follows directly from the estimates in Case I and Case II.
\end{proof}
\ \\
\textbf{Acknowledgement.} T. Mengehsa's research is partly supported by NSF grant DMS-312809.  T. Phan's research is supported by the Simons Foundation, grant \#~354889.

\end{document}